\documentclass[11pt,a4paper]{amsart}

\usepackage{layout}
\usepackage{anysize}
\marginsize{22mm}{18mm}{17mm}{18mm}
\usepackage{todonotes}
\usepackage{amsmath,amssymb,amsthm,amscd,amsfonts}
\usepackage{bbm,microtype}
\usepackage{mathtools}
\usepackage{mathrsfs}
\usepackage{hyperref}
\usepackage{graphics}
\usepackage{enumerate}
\usepackage{comment}

\theoremstyle{definition}

\newtheorem{rmk}{Remark}

\theoremstyle{plain}
\newtheorem{theorem}[rmk]{Theorem}
\newtheorem{lemma}[rmk]{Lemma}

\newtheorem{proposition}[rmk]{Proposition}
\newtheorem*{set*}{\TagSymbol}
\providecommand{\TagSymbol}{}
\title{Self-intersection of the relative dualizing sheaf on modular curves $X(N)$}

\author{Miguel Grados} 
\address{\rm Escuela de Ingeniería y Ciencias, Instituto Tecnológico y de Estudios Superiores de Monterrey, Av.~Carlos Lazo 100, Santa Fe, Álvaro Obregón, 01389 CDMX, Mexico}
\email{miguel.grados@tec.mx}
\author{Anna-Maria von Pippich}
\address{\rm Fachbereich Mathematik und Statistik, Universit\"at Konstanz, Universit\"atsstra{\upshape{\ss}}e 10, 78464 Konstanz, Germany}
\email{anna.pippich@uni-konstanz.de}
\allowdisplaybreaks

\date{\today}

\begin{document}
\setcounter{tocdepth}{1}
\setcounter{section}{0}

\begin{abstract}
Let $N\geq 3$ be a composite, odd, and square-free integer and let $\Gamma$ be the principal congruence subgroup of level $N$. Let $X(N)$ be the modular curve of genus $g_{\Gamma}$ associated to $\Gamma$. In this article, we study the Arakelov invariant $e(\Gamma)=\overline{\omega}^2/\varphi(N)$, with $\overline{\omega}^2$ denoting the self-intersection of the relative dualizing sheaf for the minimal regular model of $X(N)$, equipped with the Arakelov metric, and $\varphi(N)$ is the Euler's phi function. Our main result is the asymptotics $e(\Gamma) = 2g_{\Gamma}\log(N) + o(g_{\Gamma}\log(N))$, as the level $N$ tends to infinity.
\end{abstract}

\maketitle
%\tableofcontents

%=================%
%== SECTION ONE ==%
%=================%
\section{Introduction}

\subsection{Relevance of the known asymptotics}
%Arithmetic intersection theory is a part of modern number theory, which has its origins in the foundational work of S.~Arakelov \cite{Arakelov} on arithmetic surfaces and which was further developed by G.~Faltings \cite{Faltings} culminating in arithmetic versions of the Riemann--Roch theorem and of Noether's formula, and the proof of the Mordell conjecture. 
%Arakelov invariants, such as the Faltings's delta function and the self-intersection of the relative dualizing sheaf, play an important role in the above listed statements.
%by H.~Gillet and C.~Soul\'e \cite{GilletSoule} to general arithmetic varieties. 
%The first major achievement of Arakelov theory was the proof of 
%the Mordell conjecture by G.~Faltings, which states that a smooth projective curve of genus greater 
%than one has only finitely many rational points.

Arakelov invariants, such as the Faltings's delta function and the self-intersection of the relative dualizing sheaf, play an important role in Arakelov geometry. In particular, the self-intersection of the relative dualizing sheaf on an arithmetic surface is an essential contribution to the Faltings's height of the Jacobian \cite{MoretBailly}. It is known that suitable upper bounds for this invariant lead to important applications in arithmetic geometry. For instance, an analogue of the Bogomolov--Miyaoka--Yau inequality for arithmetic surfaces implies an effective version of the Mordell conjecture \cite{MoretBailly0}. Also, sufficiently good upper bounds for the self-intersection of the relative dualizing sheaf are crucial in the work of B.~Edixhoven and his co-authors, when estimating the running time of his algorithm for determining Galois representations \cite{Edixhoven}. However, establishing such bounds turned out to be a complicated problem. The known bounds have been established so far by the work of Abbes, Ullmo, Michel, Edixhoven, de Jong, and others, see, e.g.,~\cite{AbbesUllmo}, \cite{CurillaKuehn}, \cite{Edixhoven}, \cite{Javanpeykar}, \cite{deJong}, \cite{MichelUllmo}. They all deal with the calculation of the self-intersection of the relative dualizing sheaf for regular models of certain modular curves, Fermat curves, Belyi curves, or hyperelliptic curves, either establishing bounds or computing explicit asymptotics for this numerical invariant.
% and it measures the distance of the rational points of modular curves over number fields in the Jacobian \cite{BombieriGubler}. 

In their influential work \cite{AbbesUllmo}, A.~Abbes and E.~Ullmo considered the modular curves 
$X_{0}(N)/\mathbb{Q}$. Denoting by $\mathcal{X}_{0}(N)/\mathbb{Z}$ a minimal regular model for $X_{0}(N)/\mathbb{Q}$, and by $\overline{\omega}_{\mathcal{X}_{0}(N)/\mathbb{Z}}$ the relative dualizing sheaf on $\mathcal{X}_{0}(N)$, equipped with the Arakelov metric, they established the asymptotics 
\begin{align}
\label{bound-abbesullmo}
    \overline{\omega}_{\mathcal{X}_0(N)/\mathbbm{Z}}^2 \sim
     3g_{\overline{\Gamma}_0(N)}\log(N), 
     %+ o(g_{\overline{\Gamma}_0(N)}\log(N)),
\end{align}
%In their influential work \cite{AbbesUllmo}, A.~Abbes and E.~Ullmo 
%considered the minimal regular model $\mathcal{X}_0(N)/\mathbbm{Z}$ of  modular curves $X_0(N)$, and they obtained the following asymptotics for the self-intersection of the relative dualizing sheaf 

as $N\rightarrow\infty$, where $N$ is assumed to be square-free with $2,3\nmid N$ and such that the genus $g_{\overline{\Gamma}_0(N)}$ of $X_0(N)$ is greater than zero (see \cite{AbbesUllmo} and \cite[Th\'e{}or\`e{}me 1.1]{MichelUllmo}). As consequence of this result, P.~Michel and E.~Ullmo obtained in \cite{MichelUllmo} an asymptotics for Zhang's admissible self-intersection number \cite{Zhang}, and with this they proved an effective version of the Bogomolov conjecture for the curve $X_0(N)/\mathbbm{Q}$. Namely, if $h_{\mathrm{NT}}$ denotes the N\'e{}ron--Tate height on the Jacobian \cite{BombieriGubler}, then for all $\varepsilon>0$ and sufficiently large $N$, the set $$\{x\in X_0(N)(\overline{\mathbbm{Q}}) \mid h_{\mathrm{NT}}(x)\leq (2/3-\varepsilon)\log(N)\}$$ is finite, whereas the set $$\{x\in X_0(N)(\overline{\mathbbm{Q}}) \mid h_{\mathrm{NT}}(x)\leq (4/3+\varepsilon)\log(N)\}$$ is infinite. Recently, Banerjee--Borah--Chaudhuri \cite{BanerjeeBorahChaudhuri} removed the square-free condition on $N$ and proved that \eqref{bound-abbesullmo} and the Bogomolov conjecture hold for curves $X_0(p^2)/\mathbbm{Q}$, with $p$ a prime number.
% is infinite (see \cite[Th\'e{}or\`e{}me 1.2]{MichelUllmo}). 

By its very definition, the self-intersection of the relative dualizing sheaf on modular curves is the sum of a geometric part that encodes the finite intersection of divisors coming from the cusps, and an analytic part which is given in terms of the Arakelov Green's function evaluated at these cusps. The leading term $3g_{\overline{\Gamma}_0(N)}\log(N)$ in the asymptotics \eqref{bound-abbesullmo} is the sum of $g_{\overline{\Gamma}_0(N)}\log(N)$ that comes from the geometric part, and $2g_{\overline{\Gamma}_0(N)}\log(N)$ that comes from the analytic part. J.~Kramer conjectured that this result can be generalized to arbitrary modular curves, that is, that the main term in the asymptotics of the self-intersection of the relative dualizing sheaf for a semi-stable model of modular curves of level $N$ and genus $g$ is $3g\log(N)$, as $N$ tends to infinity. Following the lines of proof in \cite{AbbesUllmo}, Kramer's Ph.D. student H.~Mayer \cite{Mayer} obtained a positive answer to this conjecture for the case of modular curves $X_1(N)/\mathbb{Q}$. More precisely, for square-free $N$ of the form $N =N^{\prime}\cdot q\cdot r$, where $q, r> 4$ are different primes, he proved the asymptotics
 $$
 \overline{\omega}_{\mathcal{X}_1(N)/\mathbbm{Z}}^2 \sim 3g_{\overline{\Gamma}_1(N)}\log(N),
 $$ 
as $N\to\infty$; here, $\mathcal{X}_1(N)/\mathbbm{Z}$ is a minimal regular model for $X_1(N)/\mathbb{Q}$. From this asymptotics he was then able to deduce the validity of the Bogomolov conjecture in this case. 
\vskip 2mm
%the minimal regular model 

\subsection{Main result}
In this article, we establish asymptotics for the self-intersection of the relative dualizing sheaf on modular curves $X(N)$, as the level $N$ tends to infinity. Following mostly the lines of proof in \cite{AbbesUllmo}, we will first obtain a formula for this Arakelov invariant, in which the geometric and analytic parts are explicitly given. Then, we proceed to compute the asymptotics for each of these parts. Our main result is the following asymptotics (see Theorem \ref{id:main result of the paper}) 
%we present here contributes to Kramer's conjecture by showing 
$$
\frac{1}{\varphi(N)}\overline{\omega}_{\mathcal{X}(N)/\mathbbm{Z}[\zeta_N]}^2 \sim 2g_{\overline{\Gamma}(N)}\log(N),
$$ 
as $N\to\infty$, where $N\geq 3$ is a composite, odd, and square-free integer; here, $\mathbbm{Z}[\zeta_N]$ denotes the ring of integers of the cyclotomic field $\mathbb{Q}(\zeta_N)$ with a primitive $N$-th root of unity $\zeta_N$, and $\mathcal{X}(N)/\mathbb{Z}[\zeta_N]$ is a minimal regular model for $X(N)/\mathbb{Q}(\zeta_N)$. It turns out that the right hand side $2g_{\overline{\Gamma}(N)}\log(N)$ in the asymptotics comes from the analytic part, which partially confirms Kramer's conjecture. 
The difference of our result with the previous cases lies in the underlying minimal regular model, where all the computations take place, namely, our model is a $\mathbbm{Z}[\zeta_N]$-scheme which is not semi-stable. The analogue computations for a semi-stable model are a theme of our future research.

\subsection{Sketch of proof}
To prove the main result, we follow the strategy given in the work of Abbes and Ullmo \cite{AbbesUllmo}. The proof starts by observing that $\overline{\omega}^2_{\mathcal{X}(N)/\mathbb{Z}[\zeta_{N}]}/\varphi(N)$ is the sum of an explicit geometric contribution $\mathcal{G}(N)$ and an explicit analytic contribution $\mathcal{A}(N)$ (see Proposition \ref{prop:self-intersection explicit formula}). In particular, using results of Manin--Drinfeld \cite{Elkik} and Faltings--Hriljac \cite[Theorem 3.1.]{Hriljac}, we show that
\begin{align*}
	\mathcal{G}(N) &= \frac{2g_{\overline{\Gamma}(N)}(V_0,V_{\infty})_{\mathrm{fin}} - (V_0,V_0)_{\mathrm{fin}} - (V_{\infty},V_{\infty})_{\mathrm{fin}}}{2(g_{\overline{\Gamma}(N)}-1)}.
\end{align*}
%which encodes the finite intersection number of divisors coming from the cusps (?)

We refer the reader to section \ref{section6.1} for the definitions of the vertical divisors $V_0$ and $V_{\infty}$. From this, a straight-forward computation yields the asymptotics (see Proposition \ref{asymptotics:geometric part})
\begin{align*}
	\frac{1}{\varphi(N)}\mathcal{G}(N) = o(g_{\overline{\Gamma}(N)}\log(N)),
\end{align*}

as $N\to\infty$. Furthermore, the analytic contribution $\mathcal{A}(N)$ is given in terms of the 
Arakelov (or canonical) Green's function $g_{\mathrm{Ar}}(\cdot,\cdot)$, evaluated at the corresponding cusps, and equals
\begin{align*}
	\mathcal{A}(N) &= 4g_{\overline{\Gamma}(N)}(g_{\overline{\Gamma}(N)}-1)\sum_{\sigma:\mathbbm{Q}(\zeta_N) \hookrightarrow \mathbbm{C}} g_{\mathrm{Ar}}(0^{\sigma},\infty^{\sigma}).
\end{align*}

To derive the desired asymptotics for $\mathcal{A}(N)$, one first uses a fundamental identity by Abbes--Ullmo, which expresses the Arakelov Green's function $g_{\mathrm{Ar}}(q_1,q_2)$ at the cusps $q_1$ and $q_2$ of $X(N)$ in terms the scattering constant $\mathscr{C}_{q_1q_2}$ at $q_1$ and $q_2$, the hyperbolic volume $v_{\overline{\Gamma}(N)}$ of $X(N)$, the constant term in the Laurent expansion at $s=1$ of the Rankin--Selberg transforms $\mathscr{R}_{q_1}$ resp. $\mathscr{R}_{q_2}$ of the Arakelov metric at $q_1$ and $q_2$, respectively, and a contribution $\mathscr{G}$ involving the constant term of the automorphic Green's function $G_s(z,w)$ at $s=1$; we refer the reader to \eqref{dfn:Scattering constant}, \eqref{dfn:RankinSelberg constant}, and \eqref{def_mathscrG} for precise definitions.
More precisely, one has (see \eqref{id:Abbes--Ullmo formula})
\begin{align*}
	g_{\mathrm{Ar}}(q_1,q_2) = -2\pi\,\mathscr{C}_{q_1q_2} - \frac{2\pi}{v_{\overline{\Gamma}(N)}} + 2\pi(\mathscr{R}_{q_1} + \mathscr{R}_{q_2}) + 2\pi\,\mathscr{G}.
\end{align*}

The relevant scattering constants $\mathscr{C}_{q_1q_2}$ are computed, e.g., in \cite{GradosPippich}, and the asymptotics for $\mathscr{G}$ follows from bounds proven in \cite{JorgensonKramer}. Furthermore, we show that to deal with the terms $\mathscr{R}_{q_1} + \mathscr{R}_{q_2}$ for the cusps in question one is reduced to only compute $\mathscr{R}_{\infty}$ (see Lemma \ref{lem:reduction step}). To provide the explicit formula for $\mathscr{R}_{\infty}$ given in Proposition \ref{lem:R_infty}, we apply methods from the spectral theory of automorphic functions, namely, we represent $\mathscr{R}_{\infty}$ in terms of a particular automorphic kernel. Decomposing this automorphic kernel into terms involving hyperbolic and parabolic elements, $\mathscr{R}_{\infty}$ is divided into hyperbolic and parabolic contributions $\mathscr{R}_{\infty}^{\mathrm{hyp}}$ and $\mathscr{R}_{\infty}^{\mathrm{par}}$ (see identity \eqref{eqn:automorphic interpretation}). The computation of $\mathscr{R}_{\infty}^{\mathrm{hyp}}$ and $\mathscr{R}_{\infty}^{\mathrm{par}}$ are the technical heart of the paper. Note that to determine the hyperbolic contribution $\mathscr{R}_{\infty}^{\mathrm{hyp}}$, we proceed differently than Abbes--Ullmo and we reduce the calculation of the residue of the corresponding zeta function determined by hyperbolic elements to the well-known residue of the zeta function of a suitably generalized id\`ele class group of the quadratic extension. From this, it is finally shown that (see Proposition \ref{id:asymptotics of analytical contrib})
\begin{align*}
	\frac{1}{\varphi(N)}\mathcal{A}(N) = 2g_{\overline{\Gamma}(N)}\log(N) + o(g_{\overline{\Gamma}(N)}\log(N)),
\end{align*}

as $N\to\infty$. Adding up the asymptotics for the geometric and for the analytic contribution then proves the main result.
% the main term comes from the scatering constant.

%\textcolor{red}{Miguel: any idea how to add here the differnce in the stability of the model?}
%\textcolor{blue}{Idea: In the semi-stable cases (Abbes--Ullmo and Mayer), one has two irreducible components in the bad fibers and each of this components is intersected either by $H_0$ or $H_{\infty}$. In our case, we have that a bad fiber at $\mathfrak{p}$ in $\mathcal{X}/S$ has $p+1$ irreducible components, where $\mathfrak{p}|p$ with $p|N$. This difference may cause that $(V_0,V_0) = (V_{\infty},V_{\infty}) = (V_0,V_{\infty})$ for Abbes--Ullmo and Mayer, whereas for our case, we have only $(V_0,V_0) = (V_{\infty},V_{\infty})$; therefore, the $g\log(N)$ and 0 contributions for the corresponding cases.}

\subsection{Outline of the article}
The paper is organized as follows. In Section 2, we set our main notation and review basic facts on modular curves $X(N)$, non-holomorphic Eisenstein series, the spectral theory of automorphic forms, and the Arakelov metric. The core of this part is the identity \eqref{eqn:automorphic interpretation}, which states that the Rankin--Selberg constant of the Arakelov metric $\mathscr{R}_{\infty}$ at the cusp $\infty$ is essentially the sum of two contributions $\mathscr{R}_{\infty}^{\mathrm{hyp}}$ and $\mathscr{R}_{\infty}^{\mathrm{par}}$, associated to the hyperbolic and parabolic elements of $\overline{\Gamma}(N)$, respectively. In Section 3, we turn our attention to the minimal regular model $\mathcal{X}/S$ of the curve $X(N)$, where $S=\mathrm{Spec}(\mathbbm{Z}[\zeta_N])$ with $\zeta_N$ an $N$-th root of unity. Here, in Proposition \ref{prop:Model description}, we describe the good and bad fibers of $\mathcal{X}/S$, and briefly recall the moduli interpretation of the cusps of $X(N)$. Further, in Proposition \ref{prop:iota iso}, we describe how the cusps $0$ and $\infty$ are mapped on $X(N)$ under a given embedding $\mathbbm{Q}(\zeta_N) \xhookrightarrow{} \mathbbm{C}$. The importance of this result will be clear in the proof of Proposition \ref{id:asymptotics of analytical contrib}. In Section 4, we begin with the study of the zeta function associated to an hyperbolic element of $\overline{\Gamma}(N)$. In Proposition \ref{id:identity of zeta functions}, we show that this zeta function is, in fact, equal to a partial zeta function of a suitable narrow-ray ideal class, up to a holomorphic function. Using this fact, we then proceed to compute an asymptotics for the hyperbolic contribution $\mathscr{R}_{\infty}^{\mathrm{hyp}}$. This is done in Proposition \ref{prop:hyperbolic contribution}. In Section 5, we deal with the computation of the parabolic contribution $\mathscr{R}_{\infty}^{\mathrm{par}}$, and for this we proceed in the same way as in \cite{AbbesUllmo}, adapting the results therein to our case. Finally, in Section 6, we state our main result in Theorem \ref{id:main result of the paper}, and for the proof, we first provide a formula for the self-intersection of the relative dualizing sheaf in Proposition \ref{prop:self-intersection explicit formula}, where the geometric and analytic parts are explicitly given. Then, using the computations from previous sections, we determine the desired asymptotics.

\subsection{Acknowledgements}
% Idea:
We want to express our sincere gratitude to J.~Kramer, who encouraged us to work on this fascinating topic, for his patience, support, and for his guidance that led to the solution of the analytical part. In the same spirit, we are much imdebted to B.~Edixhoven for his generosity and support in developing of the results presented in Section 3 of this paper. We also want to thank P.~Bruin for his valuable comments on Section 3. Both authors acknowledge support from the International DFG Research Training Group \textit{Moduli and Automorphic Forms: Arithmetic and Geometric Aspects} (GRK 1800). Finally, the second named author would like to acknowledge support from the LOEWE research unit ``Uniformized structures in arithmetic and geometry'' of Technical University Darmstadt and Goethe-University Frankfurt.

%=================%
%== SECTION TWO ==%
%=================%
\section{Background material}

\subsection{The modular curve $X(N)$}
Let $\mathbbm{H} \coloneqq \{z=x+iy \in\mathbbm{C} \mid x,y\in\mathbb{R}; y>0\}$ be the hyperbolic upper half-plane endowed with the hyperbolic volume form $\mu_{\mathrm{hyp}}(z) \coloneqq \mathrm{d}x\,\mathrm{d}y / y^2$, and let $\mathbbm{H}^* \coloneqq \mathbbm{H} \sqcup \mathbbm{P}^{1}(\mathbbm{R})$ be the union of $\mathbbm{H}$ with its topological boundary. The group $\mathrm{PSL}_2(\mathbbm{R})$ acts on $\mathbbm{H}^*$ by fractional linear transformations. This action is transitive on $\mathbbm{H}$, since $z=x+iy=n(x)a(y)i$ with 
\begin{align*}
	n(x) \coloneqq \begin{pmatrix}1&x\\0&1 \end{pmatrix}, \quad a(y) \coloneqq \begin{pmatrix}y^{1/2} & 0 \\ 0 & y^{-1/2}\end{pmatrix}.
\end{align*}

By abuse of notation, we represent an element of $\mathrm{PSL}_2(\mathbbm{R})$ by a matrix. We set $I\coloneqq\begin{psmallmatrix} 1&0\\0&1\end{psmallmatrix}\in\mathrm{PSL}_2(\mathbbm{R})$.
Throughout the article, let $N\geq 3$ be a composite, odd, and square-free integer and let $\Gamma\coloneqq\overline{\Gamma}(N)$ denote the principal congruence subgroup of the modular group $\mathrm{PSL}_2(\mathbbm{Z})$. By $\Gamma_z \coloneqq \{\gamma\in\Gamma \mid \gamma{}z=z\}$ we denote the stabilizer subgroup of a point $z\in\mathbbm{H}^*$ with respect to $\Gamma$.
\vskip 2mm

%Let $N\geq 3$ be a composite, odd, and square-free integer and let $\Gamma\coloneqq\overline{\Gamma}(N)$ be the principal congruence subgroup of the modular group $\mathrm{PSL}_2(\mathbbm{Z})$. 
The quotient space $Y(N)\coloneqq \Gamma\backslash\mathbbm{H}$ resp.~$X(N) \coloneqq \Gamma\backslash\mathbbm{H}^*$ admits the structure of a Riemann surface and a compact Riemann surface of genus $g_{\Gamma}$, respectively. The hyperbolic volume form $\mu_{\mathrm{hyp}}(z)$ naturally defines a $(1,1)$-form on $Y(N)$, which we still denote by $\mu_{\mathrm{hyp}}(z)$, since it is $\mathrm{PSL}_2(\mathbbm{R})$-invariant on $\mathbbm{H}$. Thus, the hyperbolic volume of $Y(N)$ is given by $v_{\Gamma} \coloneqq \int_{Y(N)}\mu_{\mathrm{hyp}}(z)$. The genus $g_{\Gamma}$ and the hyperbolic volume $v_{\Gamma}$ are related by the identity
\begin{align*}
	g_{\Gamma} = 1 + \frac{v_{\Gamma}}{4\pi}\bigg( 1-\frac{6}{N} \bigg).
\end{align*}

By abuse of notation, we identify $X(N)$ with a fundamental domain $\mathcal{F}_{\Gamma}\subset\mathbbm{H}^*$; thus, at times we will identify points of $X(N)$ with their pre-images in $\mathbbm{H}^*$. 
\vskip 2mm

A cusp of $X(N)$ is the $\Gamma$-orbit of a parabolic fixed point of $\Gamma$ in $\mathbbm{H}^*$. By $P_{\Gamma}\subseteq \mathbbm{P}^1(\mathbbm{Q})$ we denote a complete set of representatives for the cusps of $X(N)$ and we write $p_{\Gamma} \coloneqq \#P_{\Gamma}$. We will always identify a cusp of $X(N)$ with its representative in $P_{\Gamma}$. Hereby, identifying $\mathbbm{P}^1(\mathbbm{Q})$ with $\mathbb{Q}\cup\{\infty\}$, we write elements of $\mathbbm{P}^1(\mathbbm{Q})$ as $\alpha/\beta$ for $\alpha,\beta\in\mathbb{Z}$, not both equal to $0$, and we always assume that $\mathrm{gcd}(\alpha,\beta)=1$; we set $1/0:=\infty$. Given a cusp $q=\alpha/\beta \in P_{\Gamma}$, we choose the scaling matrix for $q$ by $\sigma_q\coloneqq g_qa(w_q) \in \mathrm{PSL}_2(\mathbbm{R})$, where $g_q \coloneqq \begin{psmallmatrix} \alpha & * \\ \beta & *\end{psmallmatrix} \in \mathrm{PSL}_2(\mathbbm{Z})$ and $w_q \coloneqq [\mathrm{PSL}_2(\mathbbm{Z})_q : \Gamma_q]$. Let $U\subset\mathbbm{Z}$ be a set of representatives for $(\mathbbm{Z}/N\mathbbm{Z})^{\times} / \{\pm I\}$ containing $1\in\mathbbm{Z}$. Given $\xi \in U$, the element $0_{\xi} \coloneqq n(N-1)p(\xi)\gamma_N \infty \in \mathbbm{P}^1(\mathbbm{Q})$ defines a cusp of $\Gamma$, where $p(\xi) \coloneqq \begin{psmallmatrix} \xi & 1 \\ rN & \tilde{\xi} \end{psmallmatrix}$ such that $\xi\tilde{\xi} - rN = 1$ with $\tilde{\xi},r\in\mathbbm{Z}$, and $\gamma_N \coloneqq \begin{psmallmatrix}0&-1 \\ 1& 0 \end{psmallmatrix}$.

\subsection{Eisenstein series, scattering constants, and the Rankin--Selberg transform}
\label{subsect:Eisenstein Scattering RaSe}
For a given cusp $q\in P_{\Gamma}$ with scaling matrix $\sigma_q\in\mathrm{PSL}_2(\mathbbm{R})$, the non-holomorphic Eisenstein series associated to $q$ 
is defined by 
\begin{align*}
	E_{q}(z,s) \coloneqq 
\sum_{\gamma\in\Gamma_q\backslash\Gamma} \mathrm{Im}(\sigma_q^{-1}\gamma z)^s,
\end{align*}

where $z\in\mathbbm{H}$ and $s\in\mathbbm{C}$ with $\mathrm{Re}(s)>1$. The Eisenstein series
$E_{q}(z,s)$ is $\Gamma$-invariant in $z$, and holomorphic in $s$ for $\mathrm{Re}(s)>1$. Furthermore, $E_q(z,s)$ admits a meromorphic continuation to the complex $s$-plane with a simple pole at $s=1$ with residue equal to $v_{\Gamma}^{-1}$. Now, for $u\in U$, consider the subset $M(u)\subset\mathbbm{Z}^2$ given by
\begin{align}
\label{dfn:Subset M(u)}
	M(u) \coloneqq \{(m,n)\in\mathbbm{Z}^2 \mid (m,n)\equiv (0,u)\,\mathrm{mod}\,N\},
\end{align}

which is invariant under right multiplication by a matrix of $\Gamma$. Since we have $w_{\infty} = N$, $g_{\infty} = I$, and $\sigma_{\infty} = a(N)$, the Eisenstein series $E_{\infty}(z,s)$ associated to $\infty$ can be conveniently written as follows 
\begin{align*}%\label{eqn:Eisenstein prin_cong_subgrp}
	E_{\infty}(z,s) = \frac{1}{N^s}\sum_{u\in U} D_u(s)\bigg(\sum_{(m,n)\in M(u)} \frac{y^s}{|mz+n|^{2s}}\bigg),
\end{align*}

where $D_u(s)$ is the Dirichlet series defined by 
\begin{align}
\label{dfn:Dirichlet series Du(s)}
	D_u(s) \coloneqq \sum_{\substack{d=1 \\ du\equiv 1\,\mathrm{mod}\,N}}^{\infty}\frac{\mu(d)}{d^{2s}},
\end{align}

with $\mu(d)$ denoting the M\"obius function. At $s=1$, the series $D_u(s)$ admits the following Laurent expansion
\begin{align}
\label{expansion:sum of D_u(s)}
	\sum_{u\in U} D_u(s) = \frac{1}{\pi}\bigg(\frac{v_{\Gamma}}{N^3}\bigg)^{-1} + O(s-1).
\end{align}
\vskip 2mm

Let $q_1, q_2\in P_{\Gamma}$ be two cusps, not necessarily distinct, with scaling matrices $\sigma_{q_1},\sigma_{q_2}\in\mathrm{PSL}_2(\mathbbm{R})$, respectively. The scattering function $\varphi_{q_1q_2}(s)$ at the cusps $q_1$ and $q_2$ is defined by
\begin{align*}
     \varphi_{q_1q_2}(s) &\coloneqq \sqrt{\pi}\,\frac{\Gamma\left(s-\frac{1}{2}\right)}{\Gamma(s)} 
    \sum_{c=1}^{\infty}c^{-2s} 
    \,\bigg(\sum_{\substack{ d\,\mathrm{mod}\,c \\ 
    \begin{psmallmatrix}*&*\\c&d \end{psmallmatrix}\in\sigma_{q_1}^{-1} \Gamma \sigma_{q_2}}} 1 \bigg),
\end{align*}
where $s\in\mathbbm{C}$ with $\mathrm{Re}(s)>1$ and $\Gamma(s)$ denotes the Gamma function. The Eisenstein series admits the following Fourier expansion
\begin{align}
\label{id:Fourier Expansion Eisenstein}
E_{q_1}(\sigma_{q_2}z,s) = \delta_{q_1q_2}y^s + \varphi_{q_1q_2}(s)y^{1-s} + \sum_{n\neq 0}\varphi_{q_1q_2}(n;s)y^{1/2}K_{s-1/2}(2\pi|n|y)e^{2\pi{}inx},
\end{align}

where $\delta_{q_1q_2}$ is the Dirac's delta function, $K_{\mu}(Z)$ is the modified Bessel function of the second kind, and $$\varphi_{q_1q_2}(n;s) \coloneqq \frac{2\pi^s}{\Gamma(s)}|n|^{s-1/2} \sum_{c>0} c^{-2s} \bigg( \sum_{\substack{d\,\mathrm{mod}\,c \\ \begin{psmallmatrix} a&*\\c&d \end{psmallmatrix} \in \sigma_{q_1}^{-1}\Gamma\sigma_{q_2}}} e^{2\pi{}i(dm+an)/c}\bigg).$$ The scattering function $\varphi_{q_1q_2}(s)$ is holomorphic for $s\in\mathbbm{C}$ with $\mathrm{Re}(s)>1$ and admits a meromorphic continuation to the complex $s$-plane. At $s=1$, there is always a simple pole of $\varphi_{q_1q_2}(s)$ with residue equal to $v_{\Gamma}^{-1}$. The constant $\mathscr{C}_{q_1q_2}$ given by 
\begin{align}
\label{dfn:Scattering constant}
    \mathscr{C}_{q_1q_2} \coloneqq \lim_{s\rightarrow 1}\bigg( \varphi_{q_1q_2}(s) - \frac{v_{\Gamma}^{-1}}{s-1} \bigg)
\end{align}

is called the scattering constant at the cusps $q_1$ and $q_2$.
\vskip 2mm

Let $f:\mathbbm{H}\longrightarrow\mathbbm{C}$ be a $\Gamma$-invariant function which is of rapid decay at a cusp $q\in P_{\Gamma}$, i.e., the 0-th coefficient $a_0(y,q)$ in the Fourier expansion of $f(\sigma_q{}z)$ satisfies $a_0(y,q) = O(y^{-C})$ for all $C>0$, as $y\rightarrow\infty$. For $s\in\mathbbm{C}$ with $\mathrm{Re}(s)>1$, the Rankin--Selberg transform of $f$ at $q$ is given by the integral
\begin{align*}
	\mathcal{R}_q[f](s) \coloneqq \int_{\mathcal{F}_{\Gamma}}f(z) E_q(z,s) \mu_{\mathrm{hyp}}(z).
\end{align*}

The function $\mathcal{R}_q[f](s)$ possesses a meromorphic continuation to the complex $s$-plane having a simple pole at $s=1$, with residue equal to $(1/v_{\Gamma})\int_{\mathcal{F}_{\Gamma}}f(z)\mu_{\mathrm{hyp}}(z)$. Furthermore, we have the following useful identity
\begin{align*}
	\mathcal{R}_{q}[f](s) = \int_0^{\infty}a_0(y,q)y^{s-2}\mathrm{d}y.
\end{align*} 
\vskip 2mm

\subsection{The spectral expansion and automorphic kernels}
\label{subsect:Spectral Expansion}
Let $k$ be either 0 or 2 and fix it. Let $z\in\mathbbm{H}$ and $\gamma=\begin{psmallmatrix} a&b\\c&d\end{psmallmatrix}\in\mathrm{PSL}_2(\mathbbm{R})$. Given a function $f:\mathbbm{H} \longrightarrow \mathbbm{C}$, we define $f|[\gamma;k]$ by $$(f|[\gamma;k])(z) \coloneqq j_{\gamma}(z;k)^{-1}f(\gamma z),$$ where $j_{\gamma}(z;k) \coloneqq ((cz+d)/|cz+d|)^k$ is the weight-$k$ automorphy factor. A function $f:\mathbbm{H} \longrightarrow \mathbbm{C}$ is an automorphic function of weight $k$ with respect to $\Gamma$ if the equality $(f|[\gamma;k])(z) = f(z)$ holds for all $\gamma\in\Gamma$. Denote by $L^2(Y(N),k)$ the Hilbert space consisting of all automorphic functions of weight $k$ with respect to $\Gamma$ that are measurable and square integrable, endowed with the inner product given by $\langle f,g \rangle \coloneqq \int_{\mathcal{F}_{\Gamma}} f(z)\overline{g(z)} \mu_{\mathrm{hyp}}(z)$. The hyperbolic Laplacian of weight $k$ is defined by
\begin{align*}
	\Delta_{\mathrm{hyp},k} \coloneqq -y^2\bigg( \frac{\partial^2}{\partial{}x^2} + \frac{\partial^2}{\partial{}y^2}\bigg)+iky\frac{\partial}{\partial{}x}.
\end{align*}

Considering the uppering and lowering Maass operators of weight $k$, namely 
\begin{align}
\label{dfn:Uppering lowering operators}
	U_k \coloneqq \frac{k}{2} + iy\frac{\partial}{\partial{}x} + y \frac{\partial}{\partial{}y}, \qquad L_k \coloneqq \frac{k}{2} + iy\frac{\partial}{\partial{}x} - y \frac{\partial}{\partial{}y},
\end{align}

respectively, the following identities hold
\begin{align*}
	\Delta_{\mathrm{hyp},k} = L_{k+2}U_k-\frac{k}{2}\bigg(1+\frac{k}{2} \bigg) = U_{k-2}L_k + \frac{k}{2}\bigg(1-\frac{k}{2}\bigg).
\end{align*}

Since the hyperbolic Laplacian $\Delta_{\mathrm{hyp},k}$ of weight $k$ defines a symmetric and essentially self-adjoint operator, it extends to a unique self-adjoint operator $\Delta_k$ on a suitable domain. Consequently, there exists a countable orthonormal set $\{\psi_j\}_{j=0}^{\infty}$ of eigenfunctions of $\Delta_0$, in fact eigenfunctions of $\Delta_{\mathrm{hyp},0}$, such that for all $f\in L^2(Y(N),k)$, we have the spectral expansion 
\begin{align*}
%\label{id:spec decomp}
	f(z) = \sum_{j=0}^{\infty}\langle f,\psi_j\rangle \psi_j(z) + \frac{1}{4\pi}\sum_{q\in P_{\Gamma}}\int_{-\infty}^{\infty}\bigg\langle f,E_{q,k}\bigg(\cdot\,,\frac{1}{2}+ir\bigg) \bigg\rangle E_{q,k}\bigg(\cdot\,,\frac{1}{2}+ir\bigg) \mathrm{d}r,
\end{align*}

which converges in the norm topology; moreover, if $f$ is smooth and bounded, then the expansion converges uniformly on compacta of $\mathbbm{H}$. Here, $E_{q,k}(z,s)$ denotes the weight-$k$ Eisenstein series of $\Gamma$ at the cusp $q\in P_{\Gamma}$ (see, e.g., \cite[p.~291]{Roelcke}). For $k=0$, this is the Eisenstein series $E_{q}(z,s)$ from Section \ref{subsect:Eisenstein Scattering RaSe}, but if $k=2$, then $E_{q,2}(z,s)$ is determined by $E_{q}(z,s)$ via the uppering Maass operator $U_0$ given by \eqref{dfn:Uppering lowering operators}, namely 
\begin{align}
\label{id:Eisenstein between weights}
	U_0(E_{q}(z,s)) = s E_{q,2}(z,s),
\end{align}

where $s$ is not a pole of $E_{q}(z,s)$. In particular, using \eqref{id:Fourier Expansion Eisenstein} and \eqref{id:Eisenstein between weights}, the following Fourier expansion can be deduced 
\begin{align*}
%\label{id:Fourier expansion weight 2 Eisenstein}
	E_{\infty,2}(\sigma_{\infty}z,s) &= y^s + \varphi_{\infty\infty,2}(s)y^{1-s} \\[2mm] 
&+ \frac{y^{1/2}}{s} \sum_{n\neq 0} \bigg( \bigg(\frac{1}{2}-2\pi{}ny\bigg)K_{s-1/2}(2\pi|n|y) + y\frac{\partial}{\partial{}y}K_{s-1/2}(2\pi|n|y)\bigg) \varphi_{\infty\infty}(n;s)e^{2\pi{}inx}, \nonumber
\end{align*} 

where we have set $$\varphi_{\infty\infty,2}(s) \coloneqq \frac{1-s}{s}\varphi_{\infty\infty}(s).$$
\vskip 2mm

Next, we fix real numbers $T>0$ and $A>1$, once for all. Consider the holomorphic function $h_T(r)$ defined on the strip $|\mathrm{Im}(r)|<A/2$ by
\begin{align}
\label{dfn:h(r) test function}
	h_T(r) \coloneqq \exp\bigg(-T\bigg(\frac{1}{4}+r^2\bigg)\bigg).
\end{align}

Let $\phi_k$ denote the inverse Selberg--Harish--Chandra transform of weight $k$ of $h_T(r)$, namely, we have 
\begin{align*}
	\phi_k(x) \coloneqq -\frac{1}{\pi}\int_{-\infty}^{\infty}Q'(x+t^2)\left(\frac{\sqrt{x+1+t^2}-t}{\sqrt{x+1+t^2}+t}\right)^{k/2}\mathrm{d}t
\end{align*}

where $x\geq0$ and 
\begin{align*}
	Q\left(\frac{1}{4}(e^u + e^{-u}-2)\right) = \frac{1}{2}g(u), \quad g(u) = \frac{1}{2\pi}\int_{-\infty}^{\infty}h_T(r)e^{-iru}\mathrm{d}r.
\end{align*}

For $z,w\in\mathbbm{H}$, the  weight-$k$ point pair invariant $\pi_k(z,w)$ associated to $h_T(r)$ is given by
\begin{align*}
	\pi_k(z,w) \coloneqq \bigg(\frac{w-\overline{z}}{z-\overline{w}}\bigg)^{k/2}\phi_k(u(z,w)), 
\end{align*}

where $u(z,w) \coloneqq |z-w|^2/4\mathrm{Im}(z)\mathrm{Im}(w)$. With this, we define the weight-$k$ automorphic kernel by
\begin{align*}
	K_k(z,w) \coloneqq \sum_{\gamma \in \Gamma} j_{\gamma}(z;k)\pi_k(z,\gamma w).
\end{align*}

Now, let $\mathcal{S}_2(\Gamma)$ be the $\mathbbm{C}$-vector space of dimension $g_{\Gamma}$ consisting of cusp forms of weight 2 with respect to $\Gamma$, endowed with the Petersson inner product. Once for all, we fix an orthonormal basis $\{f_1, \ldots, f_{g_{\Gamma}}\}$ of $\mathcal{S}_2(\Gamma)$ and write $\lambda_j \coloneqq 1/4 + r_j^2$, with $r_j\in\mathbbm{C}$, for the corresponding eigenvalue of $\psi_j$. Then, the following spectral expansions hold
\begin{align}
\label{id:spec decom aut kernels}
\begin{split}
	K_0(z,w) &= \sum_{j=0}^{\infty}h_T(r_j)\psi_j(z)\overline{\psi}_j(w) + S_0(z,w), \\
	K_2(z,w) &= \sum_{j=1}^{g_{\Gamma}}\mathrm{Im}(z)\mathrm{Im}(w)f_j(z)\overline{f_j}(w) + \sum_{j=1}^{\infty}\frac{h_T(r_j)}{\lambda_j}(U_0\psi_j)(z)\overline{(U_0\psi_j)}(w) + S_2(z,w),
\end{split}
\end{align} 

where $U_0$ is the uppering Maass operator of weight 0 given by \eqref{dfn:Uppering lowering operators}, and 
\begin{align*}
%\label{dfn:spectral continuous part}
	S_k(z,w) \coloneqq \frac{2-k}{2}v_{\Gamma}^{-1} + \frac{1}{4\pi}\sum_{q\in P_{\Gamma}}\int_{-\infty}^{\infty}h_T(r)E_{q,k}\bigg(z,\frac{1}{2}+ir\bigg)\overline{E_{q,k}}\bigg(w,\frac{1}{2}+ir\bigg) \mathrm{d}r.
\end{align*}

In the sequel, we write $K_k(z)$ for $K_k(z,z)$ and $S_k(z)$ for $S_k(z,z)$. Letting $$\nu_k(\gamma;z) \coloneqq j_{\gamma}(z;k)\pi_k(z,\gamma z),$$ we have
$%\label{identity:automorphic kernel decomp}
	K_k(z) = K_k^{\mathrm{hyp}}(z) + K_k^{\mathrm{par}}(z)
$
with
$$
	K_k^{\mathrm{hyp}}(z) \coloneqq \sum_{\substack{\gamma\in\Gamma \\ |\mathrm{tr}(\gamma)|>2 }}\nu_k(\gamma;z), \quad
	K_k^{\mathrm{par}}(z) \coloneqq \sum_{\substack{\gamma\in\Gamma \\ |\mathrm{tr}(\gamma)| = 2 }} \nu_k(\gamma;z).
$$

Note that, by abuse of notation, the element $I$ is included in $K_{k}^{\mathrm{par}}(z)$. With this, we define 
\begin{align*}
	\mathcal{H}(z) & \coloneqq K_2^{\mathrm{hyp}}(z) - K_0^{\mathrm{hyp}}(z), \\[3mm]
	\mathcal{P}(z) & \coloneqq (K_2^{\mathrm{par}}(z) - S_2(z)) - (K_0^{\mathrm{par}}(z) - S_0(z)), \\
	\mathcal{T}(z) &\coloneqq \sum_{j=1}^{\infty}\frac{h_T(r_j)}{\lambda_j}|(U_0\psi_j)(z)|^2 - \sum_{j=0}^{\infty}h_T(r_j)|\psi_j(z)|^2.
\end{align*}

\subsection{The Arakelov metric and the Arakelov Green's function}
\label{subsect:Arakelov metric}
For $z\in \mathbbm{H}$, the Arakelov metric $F(z)$ is the $\Gamma$-invariant function defined on $\mathbbm{H}$ given by
\begin{align*}
%\label{dfn:Arakelov metric}
	F(z) \coloneqq \frac{\mathrm{Im}(z)^2}{g_{\Gamma}}\sum_{j=1}^{g_{\Gamma}}|f_j(z)|^2.
\end{align*}

It can be verified that $F(z)$ is of rapid decay at all cusps $q\in P_{\Gamma}$, and that $\int_{\mathcal{F}_{\Gamma}}F(z)\mu_{\mathrm{hyp}}(z) = 1$. Moreover, if $z\in\mathbbm{H}$ and $\alpha\in\mathrm{GL}_2^+(\mathbbm{Q})$ such that $\alpha^{-1}\Gamma\alpha = \Gamma$, then the identity $F(\alpha z) = F(z)$ holds. Following the lines of Abbes--Ullmo, one obtains the following key identity for $F(z)$. First, by definition, we have $$K_2(z) - K_0(z)= \mathcal{H}(z) + K_2^{\mathrm{par}}(z) - K_0^{\mathrm{par}}(z).$$ Using the spectral decompositions given by \eqref{id:spec decom aut kernels}, we then obtain
\begin{align*}
	K_2(z) - K_0(z) &= g_{\Gamma}F(z) + \mathcal{T}(z) + S_2(z) - S_0(z).  
\end{align*}

These two identities for the difference $K_2(z) - K_0(z)$ therefore yield the identity
\begin{align}
\label{id:spectral interpretation Arakelov metric}
	F(z) = \frac{1}{g_{\Gamma}}\bigg(-\mathcal{T}(z) + \mathcal{H}(z) + \mathcal{P}(z)\bigg).
\end{align} 

In the sequel, we refer to $\mathcal{H}(z)$ resp.~$\mathcal{P}(z)$ as the hyperbolic contribution and parabolic contribution of the Arakelov metric, respectively.
\vskip 2mm

The Rankin--Selberg constant of the Arakelov metric at $q$ is defined by
\begin{align}
\label{dfn:RankinSelberg constant}
	\mathscr{R}_{q} \coloneqq \lim_{s\rightarrow 1}\bigg( \mathcal{R}_q[F](s) - \frac{v_{\Gamma}^{-1}}{s-1}\bigg).
\end{align}

For $q=\infty$, taking the Rankin--Selberg transform on both sides of \eqref{id:spectral interpretation Arakelov metric} and grouping the constant terms of the Laurent expansion at $s=1$, we obtain the following identity
\begin{align}
\label{eqn:automorphic interpretation}
	\mathscr{R}_{\infty} = \frac{1}{g_{\Gamma}}\bigg( -\frac{v_{\Gamma}^{-1}}{2}\sum_{j=1}^{\infty}\frac{h_T(r_j)}{\lambda_j} + \mathscr{R}_{\infty}^{\mathrm{hyp}} + \mathscr{R}_{\infty}^{\mathrm{par}}\bigg).
\end{align} 

Here, $\mathscr{R}_{\infty}^{\mathrm{hyp}}$ resp.~$\mathscr{R}_{\infty}^{\mathrm{par}}$ 
%are given by \eqref{id:hyp and par contributions of R}. 
denotes the constant term in the Laurent expansion at $s=1$ of $\mathcal{R}_{\infty}[\mathcal{H}](s)$ and $\mathcal{R}_{\infty}[\mathcal{P}](s)$, respectively, i.e., we have 
\begin{align*}
%\label{id:hyp and par contributions of R}
	\begin{split}
	\mathscr{R}_{\infty}^{\mathrm{hyp}} &\coloneqq \lim_{s\rightarrow 1}\bigg(\mathcal{R}_{\infty}[\mathcal{H}](s) - \frac{v_{\Gamma}^{-1}}{s-1}\int_{\mathcal{F}_{\Gamma}}\mathcal{H}(z)\mu_{\mathrm{hyp}}(z)\bigg), \\
	\mathscr{R}_{\infty}^{\mathrm{par}} &\coloneqq \lim_{s\rightarrow 1}\bigg(\mathcal{R}_{\infty}[\mathcal{P}](s) - \frac{v_{\Gamma}^{-1}}{s-1}\int_{\mathcal{F}_{\Gamma}}\mathcal{P}(z)\mu_{\mathrm{hyp}}(z)\bigg).
	\end{split}
\end{align*}
\vskip 2mm

In the sequel, we refer to the constants $\mathscr{R}_{\infty}^{\mathrm{hyp}}$ and $\mathscr{R}_{\infty}^{\mathrm{par}}$ as the hyperbolic and parabolic contributions of $\mathscr{R}_{\infty}$, respectively.
\vskip 2mm

Next, the canonical volume form on $\mathbbm{H}$ is given by $\mu_{\mathrm{can}}(z) \coloneqq F(z) \mu_\mathrm{hyp}(z)$. Since $F(z)$ and $\mu_{\mathrm{hyp}}(z)$ are $\Gamma$-invariant on $\mathbbm{H}$, the canonical volume form $\mu_{\mathrm{can}}(z)$ is naturally defined on $Y(N)$. Furthermore, $\mu_{\mathrm{can}}(z)$ extends to a $(1,1)$-form on $X(N)$, since it remains smooth at the cusps $q\in P_{\Gamma}$. In addition, observe that $\int_{X(N)}\mu_{\mathrm{can}}(z)=1$. 
\vskip 2mm

Finally, the Arakelov Green's function $g_{\mathrm{Ar}}$ is the function defined on $X(N)\times X(N)$ which is smooth outside the diagonal and characterized by the following conditions
\begin{enumerate}[(i)]
\setlength\itemsep{0.5em}
	\item $\displaystyle \frac{1}{i\pi}\partial_z\partial_{\overline{z}}g_{\mathrm{Ar}}(z,w) = \mu_{\mathrm{can}}(z) - \delta_w(z)$,
	\item $\displaystyle \int_{X(N)}g_{\mathrm{Ar}}(z,w)\mu_{\mathrm{can}}(z) = 0$, for all $w\in X(N)$,
\end{enumerate}
where $\delta_w$ denotes the Dirac delta distribution. Given two different cusps $q_1,q_2\in P_{\Gamma}$, we have the following important identity (due to Abbes--Ullmo)
\begin{align}
\label{id:Abbes--Ullmo formula}
	g_{\mathrm{Ar}}(q_1,q_2) = -2\pi\mathscr{C}_{q_1q_2} - \frac{2\pi}{v_{\Gamma}} + 2\pi(\mathscr{R}_{q_1} + \mathscr{R}_{q_2}) + 2\pi\mathscr{G},
\end{align}

where $\mathscr{C}_{q_1q_2}$ is defined in \eqref{dfn:Scattering constant}, $\mathscr{R}_{q_j}$
is given in \eqref{dfn:RankinSelberg constant}, and 
\begin{align}
\label{def_mathscrG}
	\mathscr{G} \coloneqq -\int_{X(N)\times{}X(N)}g(z,w)\mu_{\mathrm{can}}(z) \mu_{\mathrm{can}}(w).
\end{align}

Here, the function $g(z,w)$ is the constant term in the Laurent expansion of the automorphic Green's function at $s=1$. More precisely, let $G_s(z,w)$
denote the automorphic Green's function of $\Gamma$ given, for 
$z,w\in\mathbbm{H}$, $z\not\equiv w \mod \Gamma$, and $s\in\mathbbm{C}$ with $\mathrm{Re}(s)>1$, by the series 
\begin{align*}
	G_s(z,w) = -\frac{1}{4\pi}\sum_{\gamma\in\Gamma}Q_{s-1}(1+2u(z,\gamma{}w)),
\end{align*} 

where $Q_{\tau}(\nu)$ denotes the associated Legendre function of the second kind. The automorphic Green's function $G_s(z,w)$ admits a meromorphic continuation to the whole $s$-plane
with a simple pole at $s=1$. At $s=1$, we have
\begin{align*}
G_s(z,w) = -\frac{v_{\Gamma}^{-1}}{s(s-1)} - \frac{1}{4\pi}g(z,w) + O_{z,w}(s-1),
\end{align*}

where $g(z,w)$ depends only on $z$ and $w$.
\vskip 2mm

%===================%
%== SECTION THREE ==%
%===================%
\section{The minimal regular model of the modular curve $X(N)$}

Let $K$ be a number field with ring of integers $\mathcal{O}_K$. We set $S\coloneqq \mathrm{Spec}(\mathcal{O}_K)$ and denote by $\eta$ the generic point of $S$. An arithmetic surface $\mathcal{X}/S$ is an integral, regular, and 2-dimensional $S$-scheme with a projective and flat structural morphism $\mathcal{X} \longrightarrow S$, such that the generic fiber $\mathcal{X}_{\eta}$ is a geometrically connected curve over $K$. 
\vskip 2mm

Given a smooth projective curve $X/K$ of genus $g\geq 1$ defined over $K$, there exists a unique (minimal) arithmetic surface $\mathcal{X}/S$ whose generic fiber $\mathcal{X}_{\eta}$ is isomorphic to $X/K$ \cite[Proposition 10.1.8]{Liu}. This arithmetic surface is called the minimal regular model of $X/K$. In particular, by the Riemann--Roch theorem, the compact Riemann surface $X(N)$ can be embedded in a projective space whose image is a smooth and projective algebraic curve $X(N)/\mathbbm{C}$. Since the algebraic curve $X(N)/\mathbbm{C}$ is in fact defined over the cyclotomic number field $\mathbbm{Q}(\zeta_N)$, there exists a minimal regular model of $X(N)/\mathbbm{Q}(\zeta_N)$.% which we denote by $\mathcal{X}(N)/\mathbbm{Z}[\zeta_N]$.
\vskip 2mm

To simplify notation, we will write $\mathcal{X}/S$ for the minimal regular model of $X(N)/\mathbbm{Q}(\zeta_N)$, where in this case $S = \mathrm{Spec}(\mathbbm{Z}[\zeta_N])$. Also, given an embedding $\sigma:\mathbbm{Q}(\zeta_N) \xhookrightarrow{} \mathbbm{C}$, we write $\mathcal{X}_{\eta,\sigma}$ for the base change $\mathcal{X}_{\eta} \otimes_{\sigma} \mathrm{Spec}(\mathbbm{C})$.
\vskip 2mm

Next, we will provide an explicit description of the fibers of the  arithmetic surface $\mathcal{X}/S$. To do so, let us briefly introduce the moduli interpretation of $\mathcal{X}/S$. 
\vskip 2mm

\subsection{The moduli problem of canonical structures on elliptic curves}
Let $N\geq 2$ be a given integer and set $\zeta_N \coloneqq e^{2\pi{}i/N}$. Denote by $\mu_N(\mathbbm{C})$ the group of the $N$-th roots of unity and by $Y(N)/\mathbbm{Q}(\zeta_N)$ the open subvariety of $X(N)/\mathbbm{Q}(\zeta_N)$ given by the image of $\overline{\Gamma}(N)\backslash\mathbbm{H}$ under the projective embedding that takes $X(N)$ to $X(N)/\mathbbm{C}$.
\vskip 2mm

An elliptic curve $E/T$ over an arbitrary scheme $T$ is a proper and smooth commutative group $T$-scheme with a given section, such that the geometric fibers are all connected and all have genus one. Given an elliptic curve $E/T$ over $T$, the subscheme $E[N]$ of the $N$-torsion points of $E/T$ is defined by $E[N]\coloneqq E \times_E T$, which is obtained by base change using the $N$-fold morphism $[N]: E \longrightarrow E$ and the given section of $E/T$. A canonical $\Gamma(N)$-structure on $E/T$ is a homomorphism of groups $\phi : (\mathbbm{Z}/N\mathbbm{Z})^2 \longrightarrow E[N](T)$ such that $e_N(\phi(1,0),\phi(0,1))=\zeta_N$ and the following identity of Cartier divisors holds $$\sum_{(a,b)\in (\mathbbm{Z}/N\mathbbm{Z})^2}[\phi(a,b)] = E[N];$$ here, $e_N$ denotes the Weil pairing on $E[N]$ and $[\phi(a,b)]$ is the effective Cartier divisor induced by the section $\phi(a,b)\in E[N](T)$.
\vskip 2mm

Let $R$ be a noetherian regular and excellent ring, and denote by $(\mathrm{Sch}/R)$ the category of $R$-schemes. In addition, we let $(\mathrm{Sets})$ be the category of sets. Let us consider the case when $R=\mathbbm{C}$. Given an elliptic curve $E/\mathbbm{C}$ over $\mathbbm{C}$, we have $E[N] \simeq (\mathbbm{Z}/N\mathbbm{Z})^2$, where the isomorphism is not canonical. Then, a canonical $\Gamma(N)$-structure on $E/\mathbbm{C}$ is an isomorphism $\phi : (\mathbbm{Z}/N\mathbbm{Z})^2 \longrightarrow E[N]$ which is compatible with the pairing $(\mathbbm{Z}/N\mathbbm{Z})^2 \times (\mathbbm{Z}/N\mathbbm{Z})^2 \longrightarrow \mu_N(\mathbbm{C})$ given by $((a,b),(c,d)) \longmapsto \zeta_N^{ad-bc}$ induced by the isomorphism $\phi$ and the Weil pairing on $E[N]$. Then, it turns out that the variety $Y(N)/\mathbbm{Q}(\zeta_N)$ represents the functor $\mathfrak{F} :(\mathrm{Sch}/\mathbbm{C}) \longrightarrow (\mathrm{Sets})$ which takes $T \in \mathrm{Ob}(\mathrm{Sch}/\mathbbm{C})$ to the set of isomorphism classes of pairs $(E/T,\phi)$, where $E/T$ is an elliptic curve over $T$ and $\phi$ is a canonical $\Gamma(N)$-structure on $E/T$.
\vskip 2mm

In the general case, N.~Katz and B.~Mazur proved that an analogous functor $\mathfrak{F} : (\mathrm{Sch}/\mathbbm{Z}[\zeta_N]) \longrightarrow (\mathrm{Sets})$ is representable by a flat, regular, and 2-dimensional $S$-scheme $\mathcal{Y}/S$, provided that $N\geq 3$ is a composite, odd, and square-free integer. Moreover, it turns out that the arithmetic surface $\mathcal{X}/S$ is the compactification of $\mathcal{Y}/S$. 

%obtained by the normalization of...\textcolor{red}{(check how is this construction in Katz-Mazur)}.% we have the following useful result. 

\begin{proposition}[Katz--Mazur]
\label{prop:Model description}
Let $N\geq 3$ be a composite, odd, and square-free integer, and let $\mathcal{X}/S$ be the minimal regular model of $X(N)/\mathbbm{Q}(\zeta_N)$. Then, the fiber $\mathcal{X}_{\mathfrak{p}}$, for $\mathfrak{p}\in S$  with $\mathfrak{p}\nmid N$ is a smooth curve, whereas for $\mathfrak{p}\mid p$ with $p\mid N$ a prime number, the fiber $\mathcal{X}_{\mathfrak{p}}$ consists of $p+1$ copies of smooth proper $\kappa(\mathfrak{p})$-curves intersecting transversally at their supersingular points. 
\end{proposition}

\begin{proof}
See \cite{KatzMazur}.
\end{proof}

\subsection{The moduli interpretation of the closed subscheme of cusps}
The standard $N$-gon over a scheme $S'$ is the proper curve over $S'$ (meaning a morphism $C\longrightarrow S'$ which is separated, flat, finitely presented, and of pure relative dimension 1, with non-empty fibers) obtained from $\mathbbm{P}_{S'}^1 \times \mathbbm{Z}/N\mathbbm{Z}$ by gluing the $\infty$-section of $\mathbbm{P}_{S'}^1 \times \{j\}$ to the $0$-section of $\mathbbm{P}_{S'}^1 \times \{j+1\}$, for all $j\in\mathbbm{Z}/N\mathbbm{Z}$. 
\vskip 2mm

In \cite{Conrad}, B.~Conrad extended the notion of canonical $\Gamma(N)$-structures to generalized elliptic curves. Moreover, he proved that, for $N$ sufficiently large, the scheme $\mathcal{X}/S$ represents the functor $\mathfrak{F}' :(\mathrm{Sch}/\mathbbm{Z}[\zeta_N]) \longrightarrow (\mathrm{Sets})$ which takes $T \in \mathrm{Ob}(\mathrm{Sch}/\mathbbm{Z}[\zeta_N])$ to the set of isomorphism classes of pairs $(E/T,\phi)$, where $E/T$ is now a generalized elliptic curve over $T$ and $\phi$ is an extended canonical $\Gamma(N)$-structure on $E/T$. The advantage of this approach for $\mathcal{X}/S$ lies in the moduli interpretation of the closed subscheme of cusps 
\begin{align*}
	\mathrm{Cusps}(N) \coloneqq (\mathcal{X}\, \backslash \, \mathcal{Y})^{\mathrm{red}},
\end{align*}

namely, the closed subscheme $\mathrm{Cusps}(N)$ corresponds to the extended canonical $\Gamma(N)$-structures on the Tate curve over $\mathbbm{Z}[[q^{1/N}]]$. In practice, this amounts to describe an extended canonical $\Gamma(N)$-structures on the standard $N$-gon.
\vskip 2mm

\begin{proposition}
\label{prop:iota iso}
Let $N\geq 3$ be a composite, odd, and square-free integer, and let $0,\infty \in P_{\Gamma}$ be cusps regarded as $\mathbbm{Q}(\zeta_N)$-rational points of $\mathcal{X}_{\eta}$. Given an embedding $\sigma:\mathbbm{Q}(\zeta_N) \xhookrightarrow{} \mathbbm{C}$ such that $\sigma(\zeta_N) = e^{2\pi{}iv/N}$ with $v\in(\mathbbm{Z}/N\mathbbm{Z})^{\times}/\{\pm 1\}$, let $0^{\sigma}$, $\infty^{\sigma}$ be the corresponding points in $\mathcal{X}_{\eta,\sigma}(\mathbbm{C})$ of $0,\infty$ under the embedding $\sigma$, respectively. Then, there exists an isomorphism $$\iota_{\sigma} : \mathcal{X}_{\eta,\sigma} \longrightarrow X(N)/\mathbbm{C}$$ such that, on complex points, we have $\iota_{\sigma}(0^{\sigma}) = 0_{\xi}$ for some $\xi\in U$, and $\iota_{\sigma}(\infty^{\sigma}) = \infty$.
\end{proposition}

\begin{proof}
First of all, note that, a point in $\mathcal{X}_{\eta,\sigma}(\mathbbm{C})$ corresponds to a triple $(E;P,Q)$, where $E$ is a generalized elliptic curve over $\mathbbm{C}$ and $P,Q$ is a basis of $E^{\mathrm{sm}}[N]$ such that $e_N(P,Q)=e^{2\pi{}iv/N}$. Then, $\iota_{\sigma}$ is given by $$(E;P,Q) \longmapsto (E,P+Q,v'Q),$$ on complex points.
\vskip 2mm

Secondly, for a given $\begin{psmallmatrix}a&b\\c&d \end{psmallmatrix} \in \mathrm{SL}_2(\mathbbm{Z}/N\mathbbm{Z})$, the map $$\varphi_{\begin{psmallmatrix}a&b\\c&d \end{psmallmatrix}} : (\mathbbm{Z}/N\mathbbm{Z})^2 \longrightarrow \mathbbm{Z}/N\mathbbm{Z} \times \mu_N(\mathbbm{C}),$$ defined by $(1,0)\longmapsto (a,\zeta_N^b)$ and $(0,1)\longmapsto(c,\zeta_N^d)$, describes explicitly how cusps corresponds to a $\Gamma(N)$-structure on the Tate curve. Thus, the map $\varphi_{\begin{psmallmatrix}1&0\\0&1 \end{psmallmatrix}}$ takes $(1,0)\longmapsto P_0$ and $(0,1)\longmapsto Q_0$ with $P_0 = (1,1)$ and $Q_0 = (0,1)$, and the moduli interpretation of the cusp $\infty$ is $(\mathrm{Tate}_0/\mathbbm{Q}(\zeta_N); P_0,Q_0)$. Similarly, $\varphi_{\begin{psmallmatrix}1&0\\v'&1 \end{psmallmatrix}}$ takes $(1,0)\longmapsto P_0$ and $(0,1)\longmapsto (v',\zeta_N)$, and the moduli interpretation of $1/v'$ is $(\mathrm{Tate}_0/\mathbbm{Q}(\zeta_N); P_0,(v',\zeta_N))$.
\vskip 2mm

Finally, we have 
\begin{align*}
		\iota_{\sigma}(\infty^{\sigma}) &= (\mathrm{Tate}_0; P_0+(0,e^{2\pi{}iv/N}), v'(0,e^{2\pi{}iv/N})) \\
		&= (\mathrm{Tate}_0; (1,e^{2\pi{}iv/N}),(0,e^{2\pi{}i/N})) \\
		&= (\mathrm{Tate}_0; (1,1), (0,e^{2\pi{}i/N})) = (\mathrm{Tate}_0; P_0, Q_0) = \infty,
	\end{align*}	

where in the third equality we used the automorphism $(a,z)\longmapsto(a,ze^{-2\pi{}iv/N})$ of the Tate curve. A similar reasoning can applied for the cusp 0, using the other automorphism of the Tate curve. This concludes the proof.
\end{proof}

%===================%
%== SECTION FOUR  ==%
%===================%
\section{The hyperbolic contribution of $\mathscr{R}_{\infty}$}

Recall from identity \eqref{eqn:automorphic interpretation} that $\mathscr{R}_{\infty}^{\mathrm{hyp}}$ is the hyperbolic contribution of the Rankin--Selberg constant $\mathscr{R}_{\infty}$ at $\infty$. In this part, we establish an asymptotics for $\mathscr{R}_{\infty}^{\mathrm{hyp}}$ as $T\rightarrow\infty$, for a given level $N$ (see Proposition \ref{prop:hyperbolic contribution}). To do so, in Section \ref{sect:4.1}, we first determine the residue at $s=1$ of certain zeta functions associated to hyperbolic elements of $\Gamma$ (see Proposition \ref{id:identity of zeta functions}). Next, in Section \ref{sect:4.2}, in the proof of Proposition \ref{prop:hyperbolic contribution}, we obtain an explicit identity for $\mathcal{R}_{\infty}[\mathcal{H}](s)$ in terms of the aforementioned zeta functions; thus, we derive an identity for $\mathscr{R}_{\infty}^{\mathrm{hyp}}$ from the Laurent expansion of $\mathcal{R}_{\infty}[\mathcal{H}](s)$ at $s=1$ in which the residues at $s=1$ of the previous zeta functions arise.
%\textcolor{blue}{
%In Section \ref{sect:4.1}, we first define a zeta function for a pair $(\gamma,u)$ consisting of an hyperbolic element $\gamma$ of $\Gamma$ and $u\in U$. Then, we prove that this function is equal to a partial zeta funcion of a suitable narrow-ray ideal class, up to an holomorphic factor (see Proposition \ref{id:identity of zeta functions}). Next, in Section \ref{sect:4.2}, we establish an asymptotics for $\mathscr{R}_{\infty}^{\mathrm{hyp}}$ by explicitly computing $\mathcal{R}_{\infty}[\mathcal{H}](s)$ (see Proposition \ref{prop:hyperbolic contribution}). For this computation, the integrals $J_k^+$ and $J_k^-$, and the results of Lemma \ref{lem:aux lemma hyp contrib} will be useful.
%}

\subsection{The zeta function associated to $\Gamma$}
\label{sect:4.1}
For this section, let us consider the following settings. Let $l$ be a fixed integer such that $|l|>2$ and set $\Delta \coloneqq l^2-4$. Denote by $\mathcal{Q}_{\Delta}$ the set of binary quadratic forms of discriminant $\Delta$ whose coefficients lie in $\mathbbm{Z}$. Given $f(x,y)\in\mathcal{Q}_{\Delta}$, we will write $f\cdot(x,y) \coloneqq f(y,-x)$. 
\vskip 2mm

Consider the sets 
\begin{align*}
	\mathrm{sp}_l &\coloneqq \{\gamma \in \mathrm{PSL}_2(\mathbbm{Z}) \,\mid\, \mathrm{tr}(\gamma) = l  \} \\[2mm]
\mathrm{sp}_l(\Gamma) &\coloneqq \{\gamma\in\Gamma \,|\, \mathrm{tr}(\gamma) = l \} \subset \mathrm{sp}_l.
\end{align*}

Observe that an element $\gamma\in\mathrm{sp}_l(\Gamma)$ has the form $\gamma = \gamma_l(a,b,c)$, where 
\begin{align}
\label{def:gamma_l(a,b,c)}
	\gamma_l(a,b,c) \coloneqq \begin{pmatrix} \frac{l-bN}{2} & -cN \\ aN & \frac{l+bN}{2}\end{pmatrix}
\end{align}

with $a,b,c\in\mathbbm{Z}$, $a>0$, $\mathrm{gcd}(a,b,c) = 1$, $(l\pm bN)/2 \equiv 1\,\mathrm{mod}\, N$, and $b^2-4ac$ is not a square. From now on, we assume that $\gamma\in\mathrm{sp}_l(\Gamma)$ has always the form given by \eqref{def:gamma_l(a,b,c)}. Thus, for a given $\gamma\in\mathrm{sp}_l(\Gamma)$, we set $D\coloneqq b^2-4ac$, and also
\begin{align}
\label{datum:asociated objects gamma}
	L\coloneqq \mathbbm{Q}(\sqrt{D}), \qquad \theta \coloneqq -\frac{b+\sqrt{D}}{2a}, \qquad \mathfrak{b} \coloneqq \mathbbm{Z} + \mathbbm{Z}\theta.
\end{align}

Further, if $\mathcal{O}_L$ is the ring of integers of $L$ and $\mathfrak{f}\subset\mathcal{O}_L$ is an integral ideal, we write $U(\mathfrak{f}) \coloneqq \mathcal{O}_L^{\times} \cap (1+\mathfrak{f})$, and $U_+(\mathfrak{f})$ for the totally positive elements of $U(\mathfrak{f})$, i.e., we have $$U_+(\mathfrak{f}) = \{\varepsilon \in U(\mathfrak{f}) \,|\, N(\varepsilon) > 0\},$$ where $N(\cdot)$ denotes the norm of an element in $L$. For the sequel, we set $\mathfrak{f} \coloneqq N\mathcal{O}_L$.
\vskip 2mm

It is well-known that the assignment $$\begin{pmatrix}a&b\\c&d \end{pmatrix} \longmapsto cx^2 + (d-a)xy -by^2$$ defines a bijection $\beta_l : \mathrm{sp}_l \overset{\sim}{\longrightarrow} \mathcal{Q}_{\Delta}$. For convenience, we will write $\mathcal{Q}_{\Delta}(\Gamma) \coloneqq \beta_l(\mathrm{sp}_l(\Gamma))$, and also $f_{\gamma} \coloneqq \beta_l(\gamma)$ and $\gamma_f \coloneqq \beta_l^{-1}(f)$, for $\gamma\in\mathrm{sp}_l(\Gamma)$ and $f(x,y)\in\mathcal{Q}_{\Delta}(\Gamma)$, respectively. Furthermore, $(f \circ \gamma)(x,y) \coloneqq f((x\,y)\gamma^{\mathrm{t}})$ defines an action of the modular group $\mathrm{PSL}_2(\mathbbm{Z})$ on $\mathcal{Q}_{\Delta}$, which descents to an action of $\Gamma$ on $\mathcal{Q}_{\Delta}(\Gamma)$. 
\vskip 2mm
%Therefore, we have $f_{\gamma} = N g_{\gamma}$ with $g_{\gamma}(x,y) = ax^2 + bxy + cy^2$. 

Let $\gamma=\gamma_l(a,b,c)\in\mathrm{sp}_l(\Gamma)$. Denote by $\Gamma_{f_{\gamma}}$ the stabilizer of $f_{\gamma}(x,y)\in\mathcal{Q}_{\Delta}(\Gamma)$ under the action of $\Gamma$. Since $f_{\gamma} = N g_{\gamma}$ with $g_{\gamma}(x,y) = ax^2 + bxy + cy^2$, we have $\Gamma_{f_{\gamma}} = \Gamma_{g_{\gamma}}$. Therefore, it can be verified that
\begin{align*}
	\Gamma_{f_{\gamma}} = \bigg\{ \begin{pmatrix}\frac{t-bv}{2} & -cv \\ av & \frac{t+bv}{2} \end{pmatrix} \, \bigg| \, t,v\in\mathbbm{Z}, \, t^2-Dv^2 =4, \, v\equiv 0\,\mathrm{mod}\,N, \, \frac{t-v}{2} \equiv 1\,\mathrm{mod}\,N	\bigg\}.
\end{align*}

Now, given $u\in U$ and $\gamma\in\mathrm{sp}_l(\Gamma)$, define $$M_u^{f_{\gamma}} \coloneqq \{ (m,n) \in M(u)  \mid  f_{\gamma}\cdot(m,n) > 0 \},$$ 
where $M(u)$ is given by \eqref{dfn:Subset M(u)}. We claim that the group $\Gamma_{f_{\gamma}}$ 
acts on $M_u^{f_{\gamma}}$ by right multiplication. Indeed, if $(m',n') = (m,n)\delta$ with 
$(m,n)\in M_u^{f_{\gamma}}$ and $\delta \in \Gamma_{f_{\gamma}}$, then it is easy to verify 
that $(m',n')\in M(u)$, and since $$f_{\gamma}(n',-m') = f_{\gamma}((n\,-m)(\delta^{-1})^{\mathrm{t}}) = (f_{\gamma}\circ\delta^{-1})(n,-m) = f_{\gamma}(n,-m) > 0,$$ 
the claim follows. 
\vskip 2mm

Let $\overline{M}_u^{f_{\gamma}}$ be the set of $\Gamma_{f_{\gamma}}$-orbits of $M_u^{f_{\gamma}}$, i.e., we have $\overline{M}_u^{f_{\gamma}} \coloneqq M_u^{f_{\gamma}} / \Gamma_{f_{\gamma}}$.  

\begin{lemma}
\label{lem:important bijections}
Let $u\in U$ and $\gamma\in\mathrm{sp}_l(\Gamma)$. The assignments $$u/N + n' + m'\theta \longmapsto (Nm',u+Nn') \quad \text{and} \quad \begin{pmatrix} \frac{t-bv}{2} & -cv \\ av & \frac{t+bv}{2} \end{pmatrix} \longmapsto \frac{t+v\sqrt{D}}{2}$$ define bijections $(u/N + \mathfrak{b})_+ \longrightarrow  M_u^{f_{\gamma}}$ and $\Gamma_{f_{\gamma}} \longrightarrow U_{+}(\mathfrak{f})$, respectively. Here, the set $(u/N + \mathfrak{b})_+$ consists of the elements $\xi \in u/N + \mathfrak{b}$ such that $N(\xi) > 0$. In consequence, we have 
\begin{align*}
	\overline{M}_{u}^{f_{\gamma}} \simeq (u/N + \mathfrak{b})_+ / U_+(\mathfrak{f}),
\end{align*}

where the action in the right hand side is given by $\varepsilon\cdot(u/N + n + m\theta) = u/N + \varepsilon(n+m\theta)$.
\end{lemma}

\begin{proof}
In both cases, the injectivity and surjectivity are immediate. It remains to prove the well-definition of the maps. For the first assignment, note that if $\xi = u/N + n' + m'\theta$, then 
\begin{align*}
	\xi\overline{\xi} &= \left(\frac{u}{N}+n'\right)^2 + \left(\frac{u}{N}+n'\right)\frac{b}{a}m' + \frac{c}{a}(m')^2 \\ &= \frac{1}{aN^2}\left( a(u+Nn')^2 + b(u+Nn')(Nm') + c(Nm')^2 \right) \\ &=\frac{1}{aN^3}f_{\gamma}(u+Nn',-Nm').
\end{align*}

In the last equality, we used the fact that $f_{\gamma}(x,y) = N g_{\gamma}(x,y)$ with $g_{\gamma}(x,y) = ax^2 + bxy + cy^2$. Since $N(\xi) = \xi\overline{\xi}>0$, we have $f_{\gamma}(u+Nn',-Nm') > 0$; hence, $(Nm',u+Nn')\in M_u^{f_{\gamma}}$.
\vskip 2mm

For the well-definition of the second assignment, we refer the reader to \cite[Proposition 4.2.11]{Grados}. This concludes the proof of the lemma.
\end{proof}

Next, consider a pair $(\gamma,u)$ with $\gamma\in\mathrm{sp}_l(\Gamma)$ and $u\in U$. For $s\in\mathbbm{C}$ with $\mathrm{Re}(s)>1$, the zeta function associated to $(\gamma,u)$ is defined by 
\begin{align*}
\zeta_{\gamma,u}(s) \coloneqq \sum_{(m,n)\in \overline{M}_u^{f_{\gamma}}} f_{\gamma}(n,-m)^{-s}.
\end{align*}

We claim that the function $\zeta_{\gamma,u}(s)$ possesses a meromorphic continuation to the complex plane having a pole at $s=1$ (see Proposition \ref{id:identity of zeta functions} below). To prove this claim, we recall here the definition of a partial zeta function as well as some of its main properties.
\vskip 2mm

Consider the ray $\mathfrak{m}\coloneqq \mathfrak{m}_{\infty}\mathfrak{f}$, where $\mathfrak{m}_{\infty}$ denotes the ray consisting of all embeddings $L\hookrightarrow \mathbbm{R}$, and denote by $Cl_{\mathfrak{m}}(L)$ the narrow-ray ideal class group of $L$. Given $\mathfrak{C}\in Cl_{\mathfrak{m}}(L)$ and $s\in\mathbbm{C}$ with $\mathrm{Re}(s)>1$, the partial zeta function $\zeta(s,\mathfrak{C})$ is defined by $$\zeta(s,\mathfrak{C}) \coloneqq \sum_{\substack{\mathfrak{c}\in\mathfrak{C} \\ \mathfrak{c}\neq 0}}\frac{1}{(\mathcal{N}\mathfrak{c})^s},$$ where the sum runs over the non-zero integral ideals $\mathfrak{c}$ contained in $\mathfrak{C}$ and $\mathcal{N}\mathfrak{c}$ denotes the norm of the ideal $\mathfrak{c}$. It is known that $\zeta(s,\mathfrak{C})$ defines a meromorphic function on $\mathrm{Re}>1/2$ with a unique simple pole at $s=1$ with $$\lim_{s\rightarrow 1} (s-1)\zeta(s,\mathfrak{C}) = \frac{\log(\varepsilon_{\gamma})}{N^3\sqrt{D}},$$ where $\varepsilon_{\gamma}$ is a generator of $U_+(\mathfrak{f})$ (see, e.g., \cite{Lang}). Moreover, suppose that $\mathfrak{C}\in Cl_{\mathfrak{m}}(L)$ and a nonzero integral ideal $\mathfrak{a}\in\mathfrak{C}$ are given. Let $\tau\in L^{\times}$ such that $N(\tau)>0$ and set $\mathfrak{b}' = (\tau)\mathfrak{a}^{-1}\mathfrak{f}$. Then, the partial zeta function $\zeta(s,\mathfrak{C})$ satisfies the following identity
\begin{align}
\label{id:Yamamoto}
    \zeta(s,\mathfrak{C}) = (\mathcal{N}\mathfrak{f})^{-s}\sum_{\beta\in(\tau+\mathfrak{b}')_+/U_{+}(\mathfrak{f})} \left( \frac{\mathcal{N}(\beta)}{\mathcal{N}\mathfrak{b}'} \right)^{-s}
\end{align}
(see, e.g., \cite{Yamamoto}), where $\mathcal{N}(\beta)$ denotes the norm of the principal ideal generated by the element $\beta$.

\begin{proposition}
\label{id:identity of zeta functions}
Let $u\in U$ and $\gamma\in\mathrm{sp}_l(\Gamma)$. Then, for $s\in\mathbbm{C}$ with $\mathrm{Re}(s)>1$, there exists a narrow ray ideal class $\mathfrak{C}\in Cl_{\mathfrak{m}}(L)$ such that $$\zeta_{\gamma,u}(s) = \bigg(\frac{1}{N}\bigg)^s \zeta(s,\mathfrak{C}).$$ Furthermore, $\zeta_{\gamma,u}(s)$ admits a meromorphic continuation to the complex plane having a pole at $s=1$, whose residue is given by
\begin{align*}
    \mathrm{res}_{s=1}\big(\zeta_{\gamma,u}(s)\big) = \frac{\log(\varepsilon_{\gamma})}{N^3\sqrt{D}}\,.
\end{align*}
\end{proposition}

\begin{proof}
Recall that $\gamma$ can be written as $\gamma_l(a,b,c)$ given by \eqref{def:gamma_l(a,b,c)} and associated to $\gamma$, we have the datum $\eqref{datum:asociated objects gamma}$. Now, from the proof of Lemma \ref{lem:important bijections}, we know the identity $$f_{\gamma}(u+Nn, -Nm) = aN^3N(\xi),$$ where $\xi = u/N + n+m\theta \in (u/N + \mathfrak{b})_+$. Since $\mathcal{N}\mathfrak{b} = 1/a$, $\mathcal{N}\mathfrak{f} = N^2$, and $\mathcal{N}\mathcal(\xi) = N(\xi)$, note that
\begin{align*}
	f_{\gamma}(u+nN,-Nm) =  N\cdot\mathcal{N}\mathfrak{f}\bigg(\frac{\mathcal{N}(\xi)}{\mathcal{N}\mathfrak{b}}\bigg).
\end{align*}

Therefore, we have
\begin{align*}
		\zeta_{\gamma,u}(s) &= \sum_{(m,n)\in \overline{M}_{u}^{f_{\gamma}}} f_{\gamma}(n,-m)^{-s} \\[2mm]
&= \bigg(\frac{1}{N}\bigg)^s (\mathcal{N}\mathfrak{f})^{-s}\sum_{\xi \in \left( \frac{u}{N}+ \mathfrak{b}\right)_+ / U_+(\mathfrak{f})} \bigg(\frac{\mathcal{N}(\xi)}{\mathcal{N}\mathfrak{b}}\bigg)^{-s}.
\end{align*}

Using \eqref{id:Yamamoto}, we obtain $$\zeta_{\gamma,u}(s) = \bigg(\frac{1}{N}\bigg)^s\zeta(s,\mathfrak{C}),$$ where $\mathfrak{C}$ is the narrow-ray ideal class of the fractional ideal $\mathfrak{f}\mathfrak{b}^{-1}(u/N)$. Consequently, the function $\zeta_{\gamma,u}(s)$ inherits the analytical properties of $\zeta(s,\mathfrak{C})$, in particular, the meromorphic continuation to the complex plane. This concludes the proof.
\end{proof}

\subsection{Computation of the constant $\mathscr{R}_{\infty}^{\mathrm{hyp}}$}
\label{sect:4.2}
Given $u\in U$, let us consider the following subsets of $\mathcal{Q}_{\Delta}(\Gamma) \times M(u)$
\begin{align*}
	S_{\Gamma}^{+}(u) &\coloneqq \{ (f;(m,n)) \mid (m,n)\in M_{u}^{f}, \, f\in \mathcal{Q}_{\Delta}(\Gamma)\}, \\
	S_{\Gamma}^{-}(u) &\coloneqq \{ (f;(m,n)) \mid (m,n)\in M_{u}^{-f}, \, f\in\mathcal{Q}_{\Delta}(\Gamma)\}.
\end{align*}

It can be verified that the sets $S_{\Gamma}^+(u)$ and $S_{\Gamma}^-(u)$ remain invariant under the diagonal action of $\Gamma$. Therefore, by the ``allgemeine Prinzip'' from \cite[p.~66]{ZagierZF}, we have 
\begin{align*}
S_{\Gamma}^{+}(u)/\Gamma & \simeq \bigsqcup_{\gamma\in\mathrm{sp}_l(\Gamma) / \Gamma} \overline{M}_u^{f_{\gamma}}, \\[2mm]
S_{\Gamma}^{-}(u)/\Gamma &\simeq \bigsqcup_{\gamma\in\mathrm{sp}_l(\Gamma) / \Gamma} \overline{M}_u^{(-f_{\gamma})}.
\end{align*}

Further, we have $\mathcal{Q}_{\Delta}(\Gamma) \times M(u) = S_{\Gamma}^{+}(u) \sqcup S_{\Gamma}^{-}(u)$. 
\vskip 2mm

Now, we conveniently define the integrals
\begin{align*}
	J_k^+ \coloneqq \int_{\mathcal{F}_{\Gamma}}\bigg(\sum_{(f;(m,n))\in S_{\Gamma}^{+}(u)} \nu_k(\gamma_f;z)\frac{\mathrm{Im}(z)^s}{|mz+n|^{2s}}\bigg) \mu_{\mathrm{hyp}}(z), \\
	J_k^- \coloneqq \int_{\mathcal{F}_{\Gamma}}\bigg(\sum_{(f;(m,n))\in S_{\Gamma}^{-}(u)} \nu_k(\gamma_f;z)\frac{\mathrm{Im}(z)^s}{|mz+n|^{2s}}\bigg) \mu_{\mathrm{hyp}}(z),
\end{align*}	
and the matrix $\lambda(l)$ by 
\begin{align*}
%\label{dfn:Matrix lambda(l)}
\lambda(l) \coloneqq \begin{pmatrix} l/2 & l^2/4-1 \\ 1 & l/2 \end{pmatrix}.
\end{align*}

\begin{lemma}
\label{lem:aux lemma hyp contrib}
Let $u\in U$, $l\in\mathbbm{Z}$ such that $|l|>2$, and $s\in\mathbbm{C}$ with $\mathrm{Re}(s)>1$. Then, the following identity holds
\begin{align*}
	J_k^+ &= \bigg(\int_{\mathbbm{H}}\nu_k(\lambda(l);z)\mathrm{Im}(z)^s \mu_{\mathrm{hyp}}(z)\bigg) \sum_{\gamma\in\mathrm{sp}_l(\Gamma)/\Gamma} \zeta_{\gamma,u}(s), \\
	J_k^- &= \bigg(\int_{\mathbbm{H}}\nu_k(\lambda(-l);z)\mathrm{Im}(z)^s \mu_{\mathrm{hyp}}(z)\bigg) \sum_{\gamma\in\mathrm{sp}_l(\Gamma)/\Gamma} \zeta_{\gamma,u}(s),
\end{align*}

where $\nu_k(\gamma;z)$ is given in Section \ref{subsect:Spectral Expansion}.
%Let $\Gamma=\overline{\Gamma}(N)$ and $N\geq 3$ be an odd square-free integer such that $g_{\Gamma}>1$. Suppose that $s\in\mathbbm{C}$ with $1<\mathrm{Re}(s)<\inf\{A,3/2\}$ and $l\in\mathbbm{Z}$ with $|l|>2$. Then the following identity holds
%\begin{align*}
%    \int\limits_{\mathcal{F}_{\Gamma}}^{\phantom{*}} \bigg( \sum_{\gamma\in\mathrm{sp}_l(\Gamma)} \nu_k(\gamma;z)\bigg) E(z,s) \mu_{\mathrm{hyp}}(z) = I_k(s;l) \sum_{\nu\in\Omega_{\Gamma}} D_{\Gamma,\nu}(s) \Big( \sum_{\gamma\in\mathrm{sp}_l(\Gamma)/\Gamma} \zeta_{\gamma,\nu}(s)\Big),
%\end{align*}
%\noindent where $\nu_k(\gamma;z) \coloneqq j_{\gamma}(z;k)\pi_k(z,\gamma z)$ and $I_k(s;l)$ is given by 
%\begin{align*}
%    I_k(s;l) \coloneqq \int\limits_{\mathbbm{H}}^{\phantom{*}} \Big[ \nu_k\Big(\begin{psmallmatrix}l/2&(l/2)^2-1\\1&l/2 \end{psmallmatrix};z\Big) + \nu_k\Big(\begin{psmallmatrix}-l/2&(l/2)^2-1\\1&-l/2 \end{psmallmatrix};z\Big) \Big]\mathrm{Im}(z)^s\mu_{\mathrm{hyp}}(z).
%\end{align*}
\end{lemma}

%\vskip 2mm
%Then, $S_{\Gamma}$ is invariant under the simultaneous action of $\Gamma$ in each coordinate. Furthermore, we have
%\begin{align*}
%	S_{\Gamma} / \Gamma \simeq \bigsqcup_{f\in Q_l(\Gamma)/\Gamma} \Big( M_{\Gamma}^{f} / \mathrm{Stab}_{\Gamma}(f) \Big).
%\end{align*}

\begin{proof}
For the proof of the lemma, observe that it suffices to prove the statement for the integral $J_k^+$, since we have $-f_{\gamma} = f_{-\gamma} = f_{\gamma^{-1}}$. Let $(f;(m,n))\in S_{\Gamma}^{+}(u)$ with $f(x,y) = cx^2 + (d-a)xy-by^2$, therefore $\gamma_f = \begin{psmallmatrix}a&b\\c&d \end{psmallmatrix}$. Put $(m_{\delta},n_{\delta}) = (m\, n)\delta^{\mathrm{t}}$, for $\delta\in\Gamma$, and set $$M \coloneqq \frac{1}{f(n,-m)}\begin{pmatrix}n & -bm-n(d-a)/2 \\ -m & cn-m(d-a)/2 \end{pmatrix}.$$ Using the identities $$\nu_k(\delta^{-1}\gamma\delta;z) = \nu_k(\gamma;\delta z), \qquad \frac{\mathrm{Im}(z)}{|m_{\delta}z + n_{\delta}|^2} = \frac{\mathrm{Im}(\delta z)}{|m(\delta z) + n|^{2}} ,$$ which hold for $\delta\in\mathrm{PSL}_2(\mathbbm{R})$, we obtain 
\begin{align*}
	\sum_{(f;(m,n))\in S_{\Gamma}^{+}(u)} \nu_k(\gamma_f;z)\frac{\mathrm{Im}(z)^s}{|mz+n|^{2s}} 
			&= \sum_{(f;(m,n))\in S_{\Gamma}^{+}(u)/\Gamma} \sum_{\delta\in\Gamma} \nu_k(\delta^{-1}\gamma_f\delta;z)\frac{\mathrm{Im}(z)^s}{|m_{\delta}z + n_{\delta}|^{2s}} \\ 
			&= \sum_{(f;(m,n))\in S_{\Gamma}^{+}(u)/\Gamma}\sum_{\delta\in\Gamma} \nu_k(\gamma_f;\delta z)\frac{\mathrm{Im}(\delta z)^s}{|m(\delta z) + n|^{2s}}.
\end{align*}

Thus, we have 
\begin{align*}
	J_k^+ &= \int_{\mathcal{F}_{\Gamma}} \bigg(\sum_{(f;(m,n))\in S_{\Gamma}^{+}(u)} \nu_k(\gamma_f;z) \frac{\mathrm{Im}(z)^s}{|mz+n|^{2s}}\bigg) \mu_{\mathrm{hyp}}(z) \\
		  &= \sum_{(f;(m,n))\in S_{\Gamma}^{+}(u)/\Gamma} \bigg( \int_{\mathcal{F}_{\Gamma}}\sum_{\delta\in\Gamma}\nu_k(\gamma_f;\delta z)\frac{\mathrm{Im}(\delta z)^s}{|m(\delta z) + n|^{2s}} \mu_{\mathrm{hyp}}(z) \bigg) \\ 
		  &= \sum_{(f;(m,n))\in S_{\Gamma}^{+}(u)/\Gamma} \bigg( \int_{\mathbbm{H}}\nu_k(\gamma_f;z)\frac{\mathrm{Im}(z)^s}{|mz+n|^{2s}} \mu_{\mathrm{hyp}}(z)\bigg).
\end{align*}

Finally, applying the change of variables $z\mapsto Mz$ in the last integral, and using the identities 
\begin{align*}
	\frac{\mathrm{Im}(Mz)}{|m(Mz) + n|^2} = \frac{\mathrm{Im}(z)}{f(n,-m)}, \qquad M^{-1}\gamma_f M = \lambda(l),
\end{align*}

we deduce
\begin{align*}
	J_k^+ = \bigg( \int_{\mathbbm{H}} \nu_k(M^{-1}\gamma_f M; z) \mathrm{Im}(z)^s \mu_{\mathrm{hyp}}(z) \bigg) \sum_{(f;(m,n))\in S_{\Gamma}^{+}(u)/\Gamma} f(n,-m)^{-s}.
\end{align*}

This concludes the proof of the lemma.
\end{proof}

Next, let us set
\begin{align*}
%\label{dfn:Integral I_k(s;l)}
	I_k(s;l) \coloneqq \int_{\mathbbm{H}} \bigg( \nu_k(\lambda(l);z) + \nu_k(\lambda(-l);z) \bigg)\mathrm{Im}(z)^s\mu_{\mathrm{hyp}}(z).
\end{align*}

\begin{proposition}
\label{prop:hyperbolic contribution}
Let $N\geq 3$ be an odd square-free integer and $s\in\mathbbm{C}$ with $1<\mathrm{Re}(s)<\inf\{A,3/2\}$. Then the following identity holds
\begin{align*}
    \mathcal{R}_{\infty}[\mathcal{H}](s) = \frac{1}{N^s}\sum_{\substack{l\in\mathbbm{Z} \\ |l|>2}}\big( I_2(s;l) - I_0(s;l) \big) \sum_{u\in U} \bigg(D_{u}(s) \sum_{\gamma\in\mathrm{sp}_l(\Gamma)/\Gamma} \zeta_{\gamma,u}(s)\bigg),
\end{align*}
where $D_u(s)$ is the Dirichlet series given by \eqref{dfn:Dirichlet series Du(s)}. Furthermore, we have
\begin{align*}
\mathscr{R}_{\infty}^{\mathrm{hyp}} = \frac{v_{\Gamma}^{-1}}{2}\bigg( \lim_{s\rightarrow 1}\bigg(\frac{Z_{\Gamma}'}{Z_{\Gamma}}(s) - \frac{1}{s-1}\bigg) - T + 1 + o(1)\bigg),
\end{align*}
as $T\rightarrow\infty$; here $Z_{\Gamma}(s)$ stands for the Selberg zeta function of $\Gamma$.

%\begin{align*}
%\mathscr{R}_{\infty}^{\mathrm{hyp}} = -\frac{v_{\Gamma}^{-1}}{2}\bigg(\int_{0}^{T}(\Theta_{\Gamma}(u)-1)\mathrm{d}u + T \bigg),
%\end{align*}
%with $$\Theta_{\Gamma}(u) = ,$$ where $$g_u(2\log(\eta_l))$$

%the expansion at $s=1$ of the previous integral is given by
%\begin{align*}
%    \frac{v_{\Gamma}^{-1}}{2\pi}\sum_{\substack{l\in\mathbbm{Z} \\ |l|>2}} \sum_{\gamma\in\mathrm{sp}_l(\Gamma)/\Gamma} \frac{\log(\varepsilon_{\gamma})}{\sqrt{l^2-4}}A_{|l|}(t) + O(s-1),
%\end{align*}
%where $\varepsilon_{\gamma}$ is a generator of $U_{\mathfrak{m}}$ and $A_{|l|}(t)$ is given by 
%\begin{align*}
%    A_{|l|}(t) \coloneqq -\frac{\pi}{2\eta_{l}}h\Big(\frac{i}{2}\Big) + \frac{1}{4}\int_{-\infty}^{\infty}\frac{h(r)}{\frac{1}{4}+r^2}e^{-2ir\log(\eta_{l})}\mathrm{d}r    
%\end{align*}
%with $$\eta_{l} \coloneqq \frac{|l|+\sqrt{l^2-4}}{2}.$$
\end{proposition}

\begin{proof}
For the proof, we write $K_{k}^{\mathrm{hyp}}(z)$ as follows
\begin{align*}
K_{k}^{\mathrm{hyp}}(z) = \sum_{\substack{\gamma\in\Gamma \\ |\mathrm{tr}(\gamma)|>2}} \nu_k(\gamma;z) = \sum_{\substack{l\in\mathbbm{Z} \\ |l|>2}} \sum_{\gamma \in \mathrm{sp}_l(\Gamma)} \nu_k(\gamma;z).
\end{align*}
Then, note that 
\begin{align*}
K_k^{\mathrm{hyp}}(z)E_{\infty}(z,s) = \frac{1}{N^s}\sum_{\substack{l\in\mathbbm{Z} \\ |l|>2}}\sum_{u\in U}D_u(s) \sum_{(\gamma;(m,n))}\nu_k(\gamma;z)\frac{\mathrm{Im}(z)^s}{|mz+n|^{2s}},
\end{align*}
where the innermost sum runs over the elements of $\mathrm{sp}_l(\Gamma)\times M(u)$, and since $\mathrm{sp}_l(\Gamma) \simeq \mathcal{Q}_{\Delta}(\Gamma)$, we have
\begin{align*}
K_k^{\mathrm{hyp}}(z)E_{\infty}(z,s) = \frac{1}{N^s}\sum_{\substack{l\in\mathbbm{Z} \\ |l|>2}}\sum_{u\in U}D_u(s) \sum_{(f;(m,n))}\nu_k(\gamma_f;z)\frac{\mathrm{Im}(z)^s}{|mz+n|^{2s}}.
\end{align*}
Thus, the previous identity yields
\begin{align*}
	\mathcal{R}_{\infty}[K_k^{\mathrm{hyp}}](s) &=  \int_{\mathcal{F}_{\Gamma}} K_k^{\mathrm{hyp}}(z)E_{\infty}(z,s) \mu_{\mathrm{hyp}}(z) \\
	&= \frac{1}{N^s} \sum_{ \substack{l\in\mathbbm{Z} \\ |l|>2}} \sum_{u\in U} D_u(s) \int_{\mathcal{F}_{\Gamma}} \bigg( \sum_{(f;(m,n))}\nu_k(\gamma_f;z) \frac{y^s}{|mz+n|^{2s}} \bigg) \mu_{\mathrm{hyp}}(z).
\end{align*}
Using the decomposition $\mathcal{Q}_{\Delta}(\Gamma)\times M(u) \simeq S_{\Gamma}^+(u) \sqcup S_{\Gamma}^{-}(u)$ and Lemma \ref{lem:aux lemma hyp contrib}, we get 	
\begin{align*}
\mathcal{R}_{\infty}[K_k^{\mathrm{hyp}}](s)	&= \frac{1}{N^s} \sum_{ \substack{l\in\mathbbm{Z} \\ |l|>2}} \sum_{u\in U} D_u(s) (J_k^+ + J_k^-) \\
	&= \frac{1}{N^s} \sum_{ \substack{l\in\mathbbm{Z} \\ |l|>2}} I_k(s;l) \sum_{u\in U} \bigg(D_u(s) \sum_{\gamma\in \mathrm{sp}_l(\Gamma)/\Gamma}\zeta_{\gamma,u}(s)\bigg).
\end{align*}
This proves the first assertion. 
\vskip 2mm

For the second part of the proposition, observe that
\begin{align*}
\sum_{u\in U} D_u(s)\zeta_{\gamma,u}(s) = \bigg(\sum_{u\in U}D_u(s)\bigg)\bigg(\frac{\log(\varepsilon_{\gamma})}{N^3\sqrt{D}}\frac{1}{s-1} + C + O(s-1)\bigg).
\end{align*}
Then, by \cite[Proposition 3.3.2]{AbbesUllmo}, namely 
\begin{align*}
I_2(s;l) - I_0(s;l) = \pi{}A_l(T)(s-1) + O((s-1)^2),
\end{align*}
where 
\begin{align*}
%\label{dfn:modified A_l(T)}
A_l(T) \coloneqq -\frac{1}{2\eta_{l}} + \frac{1}{4\pi}\int_{-\infty}^{\infty}\frac{h_T(r)}{\frac{1}{4}+r^2} e^{-2ir\log(\eta_{l})}\mathrm{d}r
\end{align*}
with $l> 2$ and $\eta_l \coloneqq (l+\sqrt{l^2-4})/2$, and by virtue of identity \eqref{expansion:sum of D_u(s)}, we deduce the following expansion at $s=1$
\begin{align*}
\mathcal{R}_{\infty}[\mathcal{H}](s) &= \sum_{\substack{l\in\mathbbm{Z} \\ |l|>2}}\sum_{u\in U}\sum_{\gamma\in\mathrm{sp}_l(\Gamma)/\Gamma} \frac{1}{N^s}(I_2(s;l)-I_0(s;l))D_u(s)\zeta_{\gamma,u}(s) \\
&= \sum_{\substack{l\in\mathbbm{Z} \\ |l|>2}} \sum_{\gamma\in\mathrm{sp}_l(\Gamma)/\Gamma} \frac{\log(\varepsilon_{\gamma})}{\sqrt{l^2-4}}\frac{A_l(T)}{v_{\Gamma}} + O(s-1).
\end{align*}
Therefore, we have 
\begin{align*}
\mathscr{R}_{\infty}^{\mathrm{hyp}} &= \mathcal{R}_{\infty}[\mathcal{H}](1) \\ 
&= \frac{1}{v_{\Gamma}}\sum_{\substack{l\in\mathbbm{Z} \\ |l|>2}}\sum_{\gamma\in\mathrm{sp}_l(\Gamma)/\Gamma} \frac{\log(\varepsilon_{\gamma})}{\sqrt{l^2-4}}A_l(T).
\end{align*}
Now, by \cite[Proposition 3.3.3]{AbbesUllmo}, namely
\begin{align*}
A_l(T) = -\frac{1}{2}\int_0^{T}g(t,2\log(\eta_l))\mathrm{d}t
\end{align*}
for $l>2$ with $$g(t,u) \coloneqq \frac{1}{\sqrt{4\pi{}t}}e^{-\frac{t}{4} - \frac{u^2}{4t}},$$ we get
\begin{align*}
\mathscr{R}_{\infty}^{\mathrm{hyp}} 
&= - \frac{v_{\Gamma}^{-1}}{2}\sum_{\substack{l\in\mathbbm{Z} \\ |l|>2}}\sum_{\gamma\in\mathrm{sp}_l(\Gamma)/\Gamma} \frac{\log(\varepsilon_{\gamma})}{\sqrt{l^2-4}} \int_0^{T}g(t,2\log(\eta_l))\mathrm{d}t \\
&= - \frac{v_{\Gamma}^{-1}}{2} \int_{0}^{T}\Theta_{\Gamma}(t)\mathrm{d}t,
\end{align*}
where $$\Theta_{\Gamma}(t) \coloneqq \sum_{\substack{l \in \mathbbm{Z} \\ |l|>2}}\sum_{\gamma\in\mathrm{sp}_l(\Gamma)/\Gamma}\frac{\log(\varepsilon_{\gamma})}{\sqrt{l^2-4}}g(t,2\log(\eta_l)).$$ Finally, since 
$$\int_{0}^{T}\Theta_{\Gamma}(t)\mathrm{d}t = T - \lim_{s\rightarrow 1}\bigg(\frac{Z_{\Gamma}'}{Z_{\Gamma}}(s) - \frac{1}{s-1}\bigg) - 1 +o(1)$$ holds as $T\rightarrow \infty$, the assertion follows.
\end{proof}

%==================%
%== SECTION FIVE ==%
%==================%

\section{The parabolic contribution of $\mathscr{R}_{\infty}$}

Recall from identity \eqref{eqn:automorphic interpretation} that $\mathscr{R}_{\infty}^{\mathrm{par}}$ is the parabolic contribution of the Rankin--Selberg constant $\mathscr{R}_{\infty}$ at $\infty$. In this chapter, we compute $\mathscr{R}_{\infty}^{\mathrm{par}}$ explicitly (see Proposition \ref{prop:parabolic contribution}). To do so, we first note that $$\mathcal{R}_{\infty}[\mathcal{P}](s) = \int_{0}^{\infty}p(y)y^{s-2}\mathrm{d}y,$$ where $p(y)$ is the 0-th coefficient in the Fourier expansion of $\mathcal{P}(\sigma_{\infty}z)$. Further, since $$p(y) = \sum_{j=1}^{4}(p_j(y;2) - p_j(y;0)),$$ where $p_j(y;k)$ is given by \eqref{dfn:auxiliary funcions}, it suffices to compute the integrals 
\begin{align*}
	\mathcal{M}_j(s) \coloneqq \int_{0}^{\infty} (p_j(y;2) - p_j(y;0)) y^{s-2}\mathrm{d}y,
\end{align*}
which is done in Section \ref{sect:5.1}. Then, in Section \ref{sect:5.2}, we gather the identities for $\mathcal{M}_j(s)$ and conclude the computation of $\mathscr{R}_{\infty}^{\mathrm{par}}$ in Proposition \ref{prop:parabolic contribution}.

%\textcolor{blue}{
%In this chapter, we explicitly compute $\mathscr{R}_{\infty}^{\mathrm{par}}$. To do so, we reduce the problem to determine essentially the Mellin transform $\mathcal{M}_j(s)$ of functions $p_j(y;k)$, where the  $p_j(y;k)$ are given by \eqref{dfn:auxiliary funcions}; here $j=1,2,3,4$. In Section \ref{sect:5.1}, we compute $\mathcal{M}_j(s)$, for each $j=1,2,3,4$, and in Section \ref{sect:5.2}, we compute $\mathcal{R}_{\infty}[\mathcal{P}](s)$ (see Proposition \ref{prop:parabolic contribution}).
%}
\vskip 2mm

In this section, we will consider the following settings. We define $\lambda(2) \coloneqq \begin{psmallmatrix} 1 & 0 \\ 1 & 1\end{psmallmatrix}$ and $\lambda(-2) \coloneqq \begin{psmallmatrix} -1 & 0 \\ 1 & -1\end{psmallmatrix}$, and let $I_k(s;2)$ be the integral given by
\begin{align}
\label{dfn:Integral I_k(s;2)}
	I_k(s;2) \coloneqq \int_{\mathbbm{H}}\bigg( \nu_k(\lambda(2);z) + \nu_k(\lambda(-2);z) \bigg)\mathrm{Im}(z)^s\mu_{\mathrm{hyp}}(z).
\end{align}
From \cite[Proposition 3.3.4]{AbbesUllmo}, we have the Laurent expansion at $s=1$ of the difference $I_2(s;2) - I_0(s;2)$, namely
\begin{align}
\label{exp:Expansion I_2(s;2)}
I_2(s;2) - I_0(s;2) = \sqrt{\pi}\frac{\Gamma(s-\frac{1}{2})}{\Gamma(s)}\bigg( A_2(T)(s-1) + C_1(T)(s-1)^2  + O((s-1)^3)\bigg),
\end{align}
where $$A_2(T) \coloneqq -\frac{1}{2} + \frac{1}{4\pi}\int_{-\infty}^{\infty}\frac{h_T(r)}{\frac{1}{4}+r^2}\mathrm{d}r.$$

This expansion will be useful later on. Next, recall that $\phi_k$ stands for the inverse Selberg/Harish--Chandra transform of weight $k$ of the function $h_T(r)$ given by \eqref{dfn:h(r) test function}. Then, we set $$\Psi_k^{\pm}(y) \coloneqq \Psi_k(y) + \Psi_k(-y),$$ where $\Psi_k(y)$ is the function given by
\begin{align*}
%\label{dfn:Integral Psi_k}
    \Psi_k(y) \coloneqq \int_{-\infty}^{\infty} \phi_k(u^2)\bigg( \frac{1-iu}{1+iu} \bigg)^{k/2} e^{2i\pi uy}\mathrm{d}u.
\end{align*}

We will write $$\widetilde{E}_{\infty,k}(\sigma_{\infty}z,s) \coloneqq E_{\infty,k}(\sigma_{\infty}z,s) - y^s - \varphi_{\infty\infty,k}(s),$$ where $E_{\infty,0}(z,s)$ and $\varphi_{\infty\infty,0}(s)$ stands for $E_{\infty}(z,s)$ and $\varphi_{\infty\infty}(s)$, respectively. For convenience, we define the following integrals
\begin{align}
\label{dfn:auxiliary funcions}
   p_1(y;k) &\coloneqq \int_{-1/2}^{1/2}\Big( \sum_{\substack{\gamma\in\Gamma\,\setminus\,\Gamma_{\infty} \nonumber \\ |\mathrm{tr}(\gamma)|=2}} \nu_{k}(\gamma;\sigma_{\infty}z)\Big) \mathrm{d}x;\nonumber \\
   p_2(y;k) &\coloneqq \bigg( \sum_{\gamma\in\Gamma_{\!_{\infty}}}\nu_{k}(\gamma;\sigma_{\infty}z) \bigg)  - \frac{y}{2\pi} \int_{-\infty}^{\infty}h_T(r)\mathrm{d}r;\nonumber\\
   p_3(y;k) &\coloneqq - \frac{y}{2\pi}\int_{-\infty}^{\infty}h_T(r)\bigg( \frac{\frac{1}{2}+ir}{\frac{1}{2}-ir} \bigg)^{k/2}\varphi_{\infty\infty}\bigg( \frac{1}{2}-ir\bigg) y^{2ir} \mathrm{d}r - \frac{2-k}{2}v_{\Gamma}^{-1}; \\[2mm]
   p_4(y;k) &\coloneqq - \int_{-1/2}^{1/2} \bigg( \frac{1}{4\pi} \int_{-\infty}^{\infty} h_T(r) \sum_{q\in P_{\Gamma}} \bigg| \widetilde{E}_{\infty,k}\bigg(\sigma_{\infty}z,\frac{1}{2}+ir\bigg)\bigg|^2\mathrm{d}r\bigg)\mathrm{d}x; \nonumber\\[2mm]
    p_4^*(y) &\coloneqq - \frac{1}{4\pi} \sum_{q\in P_{\Gamma}} \int_{-\infty}^{\infty} \bigg( \frac{h_T(r)}{\frac{1}{4}+r^2} \mathcal{R}_{\infty}\bigg[ \, \bigg| \widetilde{E}_{\infty,k}\bigg(\sigma_{\infty}z,\frac{1}{2}+ir\bigg)\bigg|^2\bigg](s) \bigg) \mathrm{d}r; \nonumber
\end{align}
%\noindent where $\widetilde{E}_{\infty,k}(\sigma_{\infty}z,s)$ is given by $$\phantom{\int}\widetilde{E}_{\infty,k}(\infty_{\infty}z,s) \coloneqq E_{\infty,k}(\sigma_{\infty}z,s) - y^s -\varphi_{\infty\infty,k}(s)y^{1-s}.$$ 
where the sum in the integrand of the function $p_1(y;k)$ runs over all elements of $\Gamma$ that are not in $\Gamma_{\infty}$. 

%Finally, for $j=1,2,3,4$, we define 
%\begin{align*}
%	\mathcal{M}_j(s) \coloneqq \int_{0}^{\infty} (p_j(y;2) - p_j(y;0)) y^{s-2}\mathrm{d}y.
%\end{align*}

%Next, we proceed to establish new notation and prove some necessary lemmas.

\subsection{Preparatory lemmas}
\label{sect:5.1}

\begin{lemma}
Let $N\geq 3$ be an odd square-free integer and $s\in\mathbbm{C}$ with $1<\mathrm{Re}(s)<A$. Then, the following identity holds
\begin{align*}
	\mathcal{M}_1(s) &= \frac{2\zeta(s)\zeta(2s-1)}{\zeta(2s)N^{2s-1}}(I_2(s;2) - I_0(s;2)),
\end{align*}
where $\zeta(s)$ denotes the Riemann zeta funcion and $I_k(s;2)$ is given by \eqref{dfn:Integral I_k(s;2)}. Furthermore, the Laurent expansion of $\mathcal{M}_1(s)$ at $s=1$ is given by
\begin{align*}
	\frac{12A_2(T)}{\pi N}\frac{1}{s-1} + \frac{12}{\pi N}\bigg(C_1(T) + A_2(T)(2\mathscr{C}+\gamma_{\!_{\mathrm{EM}}}-2\log(N)) \bigg) + O(s-1),
\end{align*}
where $\gamma_{\!_{\mathrm{EM}}}$ denotes the Euler--Mascheroni constant and $C_1(T)$ is a constant that depends only on the fixed positive real $T$.% and is such that $\lim_{T \rightarrow\infty}C_1(T)<\infty$, and $A_2(T)$ is given by $$A_2(T) \coloneqq - \frac{1}{2} + \frac{1}{4\pi}\int_{-\infty}^{\infty}\frac{h_T(r)}{\frac{1}{4}+r^2}\mathrm{d}r.$$
\end{lemma}

\begin{proof}
By definition, we have $$\mathcal{M}_1(s) = \int_{0}^{\infty}(p_1(y;2) - p_1(y;0))y^{s-2}\mathrm{d}y.$$ In order to compute $\mathcal{M}_1(s)$, let us consider the following integral $$\int_0^{\infty}p_1(y;k)y^{s-2}\mathrm{d}y = \int_0^{\infty}\int_{-1/2}^{1/2}\bigg( \sum_{\substack{\gamma\in\Gamma\,\setminus\,\Gamma_{\infty} \\ |\mathrm{tr}(\gamma)|=2}} \nu_k(\gamma;\sigma_{\infty}z) \bigg)y^{s-2}\mathrm{d}x\mathrm{d}y.$$ 

Put $B \coloneqq \{n(b) \mid b\in\mathbbm{Z}\}$. Then, note that we have bijections
\begin{align*}
	N\mathbbm{Z}\setminus\{0\} \times (B\backslash\mathrm{PSL}_2(\mathbbm{Z})) & \overset{\sim}{\longrightarrow } \{\gamma\in\Gamma\setminus\Gamma_{\infty} \mid \mathrm{tr}(\gamma) = 2 \},  \\[2mm]
	N\mathbbm{Z}\setminus\{0\} \times -(B\backslash\mathrm{PSL}_2(\mathbbm{Z})) & \overset{\sim}{\longrightarrow} \{\gamma\in\Gamma\setminus\Gamma_{\infty} \mid \mathrm{tr}(\gamma) = -2 \},
\end{align*}
given by 
\begin{align*}
	\bigg(m,B\begin{pmatrix} *&*\\c&d \end{pmatrix}\bigg) & \longmapsto \gamma_+(m;c,d) \coloneqq \begin{pmatrix} 1+mcd & md^2 \\ -mc^2 & 1-mcd \end{pmatrix}, \\
	\bigg(m,B\begin{pmatrix} *&*\\c&d \end{pmatrix}\bigg) & \longmapsto \gamma_-(m;c,d) \coloneqq \begin{pmatrix} -1+mcd & md^2 \\ -mc^2 & -1-mcd \end{pmatrix},
\end{align*}
respectively. With these bijections, we obtain
\begin{align*}
	\int_0^{\infty} p_1(y;k)y^{s-2} \mathrm{d}y &= \int_0^{\infty} \int_{-1/2}^{1/2} \bigg( \sum_{\left(m,B\begin{psmallmatrix}*&*\\c&d \end{psmallmatrix}\right)}^{+}\nu_k(\gamma_+(m;c,d);\sigma_{\infty}z) \\
	&+ \sum_{\left(m,B\begin{psmallmatrix}*&*\\c&d \end{psmallmatrix}\right)}^{-}\nu_k(\gamma_-(m;c,d);\sigma_{\infty}z) \bigg)y^{s-2}\mathrm{d}x\mathrm{d}y,
\end{align*}

where the first sum runs over all pairs $(m,B \begin{psmallmatrix}*&*\\c&d \end{psmallmatrix})$ in $N\mathbbm{Z}\setminus\{0\}\times(B\backslash\mathrm{PSL}_2(\mathbbm{Z}))$, and the second sum runs over all pairs $(m,B\begin{psmallmatrix}*&*\\c&d \end{psmallmatrix})$ in $N\mathbbm{Z}\setminus\{0\} \times -(B\backslash\mathrm{PSL}_2(\mathbbm{Z}))$. Then it suffices to compute 
\begin{align*}
	P_k^+ \coloneqq \int_0^{\infty} \int_{-1/2}^{1/2} \bigg( \sum_{\left(m,B\begin{psmallmatrix}*&*\\c&d \end{psmallmatrix}\right)}^{+}\nu_k(\gamma_+(m;c,d);\sigma_{\infty}z)\bigg) y^{s-2}\mathrm{d}x\mathrm{d}y.
\end{align*}

Observe that, 
\begin{align*}
	P_k^+ = \int_0^{\infty} \int_{-1/2}^{1/2} \bigg( \sideset{}{'}\sum_{\left(m,B\begin{psmallmatrix}*&*\\c&d \end{psmallmatrix}\right)} \sum_{n\in\mathbbm{Z}} \nu_k(\gamma_+(m;c,d);\sigma_{\infty}z)\bigg) y^{s-2}\mathrm{d}x\mathrm{d}y,
\end{align*}
where now, the first sum runs over all pairs $(m,B \begin{psmallmatrix}*&*\\c&d \end{psmallmatrix})$ representing a class in $N\mathbbm{Z}\setminus\{0\}\times(B\backslash\mathrm{PSL}_2(\mathbbm{Z})/\Gamma_{\infty})$; further, by convergence, we can rewrite $P_k^+$ as follows
\begin{align*}
	P_k^+ &= \sideset{}{'}\sum_{\left(m,B\begin{psmallmatrix}*&*\\c&d \end{psmallmatrix}\right)} \int_0^{\infty} \bigg( \sum_{n\in\mathbbm{Z}} \int_{-1/2}^{1/2} \nu_k(\gamma_+(m;c,d);N(x+n) + iNy)\mathrm{d}x \bigg) y^{s-2}\mathrm{d}y.
\end{align*}

Now, by a suitable change of variables, we obtain
\begin{align*}
P_k^+ &= \frac{1}{N}\sideset{}{'}\sum_{\left(m,B\begin{psmallmatrix}*&*\\c&d \end{psmallmatrix}\right)} \int_0^{\infty}\int_{-\infty}^{\infty} \nu_k(\gamma_+(m;c,d);x+iNy) y^{s-2}\mathrm{d}x \mathrm{d}y \\
&= \frac{1}{N^s} \sideset{}{'}\sum_{\left(m,B\begin{psmallmatrix}*&*\\c&d \end{psmallmatrix}\right)} \int_{\mathbbm{H}} \nu_k(\gamma_+(m;c,d);z) \mathrm{Im}(z)^s \mu_{\mathrm{hyp}}(z).
\end{align*}
Then, by \cite[Lemme 3.2.12]{AbbesUllmo}, namely
\begin{align*}
\int_{\mathbbm{H}} \nu_k(\gamma_+(m;c,d);z) \mathrm{Im}(z)^s\mu_{\mathrm{hyp}(z)} = \frac{1}{|mc^2|^s}\int_{\mathbbm{H}}\nu_k(\lambda(2);z)\mathrm{Im}(z)^s\mu_{\mathrm{hyp}}(z), 
\end{align*}
we get
\begin{align*}
P_k^+ = \frac{1}{N^s} \bigg( \int_{\mathbbm{H}} \nu_k(\lambda(2);z) \mathrm{Im}(z)^s \mu_{\mathrm{hyp}}(z) \bigg) \sideset{}{'}\sum_{\left(m,B\begin{psmallmatrix}*&*\\c&d \end{psmallmatrix}\right)} \frac{1}{|mc^2|^s},
\end{align*}

Similarly, if we set 
\begin{align*}
	P_k^- \coloneqq \int_0^{\infty} \int_{-1/2}^{1/2} \bigg( \sum_{\left(m,B\begin{psmallmatrix}*&*\\c&d \end{psmallmatrix}\right)}^{-}\nu_k(\gamma_+(m;c,d);\sigma_{\infty}z)\bigg) y^{s-2}\mathrm{d}x\mathrm{d}y,
\end{align*}
we will obtain 
\begin{align*}
P_k^- = \frac{1}{N^s} \bigg( \int_{\mathbbm{H}} \nu_k(\lambda(-2);z) \mathrm{Im}(z)^s \mu_{\mathrm{hyp}}(z) \bigg) \sideset{}{'}\sum_{\left(m,B\begin{psmallmatrix}*&*\\c&d \end{psmallmatrix}\right)} \frac{1}{|mc^2|^s}.
\end{align*}
Therefore, 
\begin{align*}
\int_0^{\infty} p_1(y;k)y^{s-2} \mathrm{d}y = \frac{1}{N^s} I_k(s;2) \sideset{}{'} \sum_{\left(m,B\begin{psmallmatrix}*&*\\c&d \end{psmallmatrix}\right)} \frac{1}{|mc^2|^s},
\end{align*}
and since
\begin{align*}
	 \sideset{}{'} \sum_{\left(m,B\begin{psmallmatrix}*&*\\c&d \end{psmallmatrix}\right)} \frac{1}{|mc^2|^s} 
&= \sum_{m\in N\mathbbm{Z}\setminus\{0\}} \frac{1}{|m|^s} \bigg( \sum_{c=1}^{\infty} \frac{1}{c^{2s}} \sum_{\substack{d\,\mathrm{mod}\,c \\ \begin{psmallmatrix} *&*\\c&d \end{psmallmatrix} \in B\backslash\mathrm{PSL}_2(\mathbbm{Z})/\Gamma_{\infty}}} 1 \bigg)\\
&= \frac{2\zeta(s)\zeta(2s-1)}{\zeta(2s)N^{s-1}},
\end{align*}
it follows that 
\begin{align*}
	\mathcal{M}_1(s) = 2N^{1-2s}\frac{\zeta(s)\zeta(2s-1)}{\zeta(2s)}(I_2(s;2) - I_0(s;2)).
\end{align*}
This proves the first assertion of the lemma. For the second part of the lemma, use \eqref{exp:Expansion I_2(s;2)} together with the well-known Laurent expansion
\begin{align*}
\sqrt{\pi}\frac{\Gamma(s-1/2)}{\Gamma(s)}\frac{\zeta(2s-1)}{\zeta(2s)} = \frac{3/\pi}{s-1} + \frac{6}{\pi}\mathscr{C} + O(s-1)
\end{align*}
at $s=1$ with $\mathscr{C} \coloneqq 1 - \log(4\pi) + \zeta'(-1)/\zeta(-1)$. This concludes the proof.
\end{proof}

\begin{lemma}
Let $N\geq 3$ be an odd square-free integer and $s\in\mathbbm{C}$ with $1<\mathrm{Re}(s)<A$. Then, the following identity holds
\begin{align*}
    \mathcal{M}_2(s) = 2\zeta(s)\bigg( \int_{0}^{\infty}\Psi_2^{\pm}(2y) y^{s-1}\mathrm{d}y - \int_{0}^{\infty} \Psi_0^{\pm}(2y) y^{s-1}\mathrm{d}y  \bigg).
\end{align*}

Furthermore, the Laurent expansion of $\mathcal{M}_2(s)$ at $s=1$ is given by
\begin{align*}
    \frac{C_2(T) + (4\pi)^{-1}}{s-1} + \bigg(\frac{1- \log(4\pi)}{4\pi} + \gamma_{\!_{\mathrm{EM}}}C_2(T) + C_3(T)\bigg) + O(s-1),
\end{align*}
where $C_2(T)$ and $C_3(T)$ are constants that depend only on the fixed positive real $T$.%, and both tend to 0 as $T\rightarrow\infty$.
\end{lemma}

\begin{proof}
By definition, we have
\begin{align*}
	\mathcal{M}_2(s) = \int_0^{\infty} (p_2(y;2) - p_2(y;0))y^{s-2}\mathrm{d}y.
\end{align*}
In order to compute $\mathcal{M}_2(s)$, let us consider the integral 
\begin{align*}
	\int_0^{\infty}p_2(y;k)y^{s-2}\mathrm{d}y = \int_0^{\infty}\bigg( \sum_{\gamma\in\Gamma_{\infty}}\nu_{k}(\gamma,\sigma_{\infty}z) - \frac{y}{2\pi}\int_{-\infty}^{\infty}h_T(r)\mathrm{d}r\bigg)y^{s-2}\mathrm{d}y.
\end{align*}

We claim that
\begin{align*}
\sum_{\gamma\in\Gamma_{\infty}} \nu_k(\gamma;\sigma_{\infty}z) = \frac{y}{2\pi}\int_{-\infty}^{\infty} h_T(r)\mathrm{d}r + 2y \sum_{\substack{n\in\mathbbm{Z} \\ n\neq 0}} \Psi_k(2ny).
\end{align*}

Indeed, since $\gamma\in\Gamma_{\infty}$ has the form $\gamma=\begin{psmallmatrix} 1 & nN \\ 0 & 1 \end{psmallmatrix}$, for some $n\in\mathbbm{Z}$, we have  $j_{\gamma}(z;k)=1$. This implies 
\begin{align*}
\nu_k(\gamma;\sigma_{\infty}z) &= \pi_k(Nz,N(z+b)) \\[2mm]
&= \bigg( \frac{1-i(n/2y)}{1+i(n/2y)} \bigg)^{k/2}\phi_k\bigg(\frac{n^2}{4y^2}\bigg),
\end{align*}
where for the second equality, we used the definition of $\pi_k(z,w)$ and $u(z,w)$ given in Section \ref{subsect:Spectral Expansion}. Therefore, we get the identities
\begin{align*}
\sum_{\gamma\in\Gamma_{\infty}} \nu_k(\gamma;\sigma_{\infty}z) &= \sum_{n\in\mathbbm{Z}} \bigg(\frac{1-i(n/2y)}{1+i(n/2y)}\bigg)^{k/2}\phi_k\bigg(\frac{n^2}{4y^2}\bigg) \\
&= \sum_{n\in\mathbbm{Z}} \int_{-\infty}^{\infty}\phi_k\bigg(\frac{v^2}{4y^2}\bigg) \bigg(\frac{1-i(v/2y)}{1+i(v/2y)}\bigg)^{k/2} e^{2\pi i nv} \mathrm{d}v,
\end{align*}
where in the last equality we applied the Poisson summation formula. The claim follows by noting that $$\int_{-\infty}^{\infty}\phi_k(v^2) \bigg(\frac{1+iv}{1-iv}\bigg)^{k/2} \mathrm{d}v = \frac{1}{4\pi}\int_{-\infty}^{\infty} h_T(r)\mathrm{d}r.$$

Consequently, we have 
\begin{align*}
\int_0^{\infty}p_2(y;k)y^{s-2}\mathrm{d}y &= 2\sum_{\substack{n\in\mathbbm{Z} \\ n\neq 0}} \int_{0}^{\infty} \Psi_k(2ny)y^{s-1}\mathrm{d}y \\
&= 2\sum_{n=1}^{\infty}\int_{0}^{\infty}( \Psi_k(2ny) + \Psi_k(-2ny) )y^{s-1}\mathrm{d}y \\
&= 2\zeta(s)\int_0^{\infty}\Psi_k^{\pm}(2y)y^{s-1}\mathrm{d}y,
\end{align*}
and the first assertion of the lemma follows. For the second assertion, note that $\Psi_k^{\pm}(2y)$ is one-half the function $\psi(y) + \psi(-y)$ given by \cite[p.~42]{AbbesUllmo} when $k=2$, and given by \cite[p.~326--329]{Zagier} when $k=0$. Finally, the assertion of the lemma follows using \cite[p.~389]{Zagier} and \cite[Proposition 3.2.6]{AbbesUllmo}. This concludes the proof.
%
%Thus, by writing $z=x+iy$, we obtain
%\begin{align*}
%	\sum_{\gamma\in\Gamma_{\infty}}\nu_{k}(\gamma,\sigma_{\infty}z) = \sum_{b\in\mathbbm{Z}}\bigg( \frac{2-i(bN/y)}{2+i(bN/y)} \bigg)^{k/2} \phi_k\bigg(\bigg(\frac{bN}{2}\bigg)^2\bigg).
%\end{align*}
%Now, by applying the Poisson summation formula in the right hand side of the previous identity, we get 
%\begin{align*}
%	\sum_{\gamma\in\Gamma_{\infty}}\nu_{k}(\gamma,\sigma_{\infty}z) &= \sum_{n\in\mathbbm{Z}}\int_{-\infty}^{\infty} \phi_k\bigg( \bigg(\frac{vN}{y}\bigg)^2 \bigg) \bigg( \frac{2-i(vN/y)}{2+i(vN/y)} \bigg)^{k/2} e^{2\pi inv}\mathrm{d}v \\ 
%	&=\frac{y}{N}\sum_{n\in\mathbbm{Z}} \int_{-\infty}^{\infty}\phi_k(u^2)\bigg(\frac{2+iu}{2-iu}\bigg)^{k/2} e^{2\pi{}inuy/N}\mathrm{d}u. 
%\end{align*}
%
%Further, if we split the sum into $n=0$ and $n\neq 0$, then we obtain 
%\begin{align*}
%\sum_{\gamma\in\Gamma_{\infty}}\nu_{k}(\gamma,\sigma_{\infty}z) &= \frac{y}{2\pi}\frac{1}{N}\int_{-\infty}^{\infty}h_T(r)\mathrm{d}r + \frac{y}{N}\sum_{\substack{n\in\mathbbm{Z} \\ n\neq 0}} \Psi_k\bigg(\frac{ny}{N}\bigg),
%\end{align*}
%
%and consequently, we deduce that 
%
%\begin{align*}
%	\int_0^{\infty}p_2(y;k)y^{s-2}\mathrm{d}y &= \frac{1}{N}\sum_{n=1}^{\infty} \int_0^{\infty}\bigg(\Psi_k\bigg(\frac{ny}{N}\bigg) + \Psi_k\bigg(-\frac{ny}{N}\bigg)\bigg)y^{s-2}\mathrm{d}y \\ 
%	&= N^{s-1}\zeta(s)\int_0^{\infty}\Psi_{k}^{\pm}(y)y^{s-1}\mathrm{d}y
%\end{align*}
\end{proof}

\begin{lemma}
Let $N\geq 3$ be an odd square-free integer and $s\in\mathbbm{C}$ with $1<\mathrm{Re}(s)<A$. Then, the following identity holds
\begin{align*}
    \mathcal{M}_3(s) = \bigg( \frac{s}{1+s} \bigg) h_T\bigg(\frac{is}{2}\bigg)\varphi_{\infty\infty}\bigg( \frac{1+s}{2} \bigg).
\end{align*}
\noindent Furthermore, the Laurent expansion of $\mathcal{M}_3(s)$ at $s=1$ is given by
\begin{align*}
    \frac{v_{\Gamma}^{-1}}{s-1} + \bigg( \frac{\mathscr{C}_{\infty\infty}}{2} + \frac{v_{\Gamma}^{-1}}{2}(T+1) \bigg) + O(s-1).
\end{align*}
\end{lemma}

\begin{proof}
The proof of the first identity follows from an immediate extension of \cite[Lemme 3.2.17, p.~44]{AbbesUllmo} and \cite[Lemma 5.1.1., p.~137]{Mayer} to subgroups of the modular group of finite index. This is possible since it only uses general analytic properties of the scattering function. Next, for the Laurent expansion at $s=1$, we just multiply the following Laurent expansions at $s=1$
\begin{flalign*}
    & \frac{1}{2}h_T\bigg( \frac{is}{2} \bigg) = \frac{1}{2} + \frac{T}{4}(s-1) + O\big( (s-1)^2 \big); \\[2mm]
    & \varphi_{\infty\infty}\bigg( \frac{1+s}{2} \bigg) = \frac{2v_{\Gamma}^{-1}}{s-1} + \mathscr{C}_{\infty\infty} + O(s-1); \\[2mm]
    & \frac{2s}{s+1} = 1 + \frac{1}{2}(s-1) + O\big( (s-1)^2 \big).
\end{flalign*}
This concludes the proof.
\end{proof}

\begin{lemma}
Let $N\geq 3$ be an odd square-free integer and $s\in\mathbbm{C}$ with $1<\mathrm{Re}(s)<A$. Then, the following identity holds
\begin{align*}
    \mathcal{M}_4(s) = s\bigg(\frac{s-1}{2}\bigg)\int_{0}^{\infty}p_4^*(y)y^{s-2}\mathrm{d}y,
\end{align*}
\noindent where $p_4^*(y)$ is given by \eqref{dfn:auxiliary funcions}. Furthermore, the Laurent expansion of $\mathcal{M}_4(s)$ at $s=1$ is given by
\begin{align*}
    C_4(T) + O(s-1),
\end{align*}
where $C_4(T)$ is a constant that depends only on the fixed positive real $T$.%, and tends to 0 as $T\rightarrow\infty$.
\end{lemma}

\begin{proof}
The lemma follows from an immediate extension of \cite[Proposition 5.2.3, p.~141]{Mayer} to subgroups of the modular group of finite index.
\end{proof}

\subsection{Computation of the constant $\mathscr{R}_{\infty}^{\mathrm{par}}$}
\label{sect:5.2}

\begin{proposition}
\label{prop:parabolic contribution}
Let $N\geq 3$ be an odd square-free integer and $s\in\mathbbm{C}$ with $1<\mathrm{Re}(s)<\inf\{A,3/2\}$. Then the following identity holds
\begin{align*}
    \mathcal{R}_{\infty}[\mathcal{P}](s) &= \frac{2\zeta(s)\zeta(2s-1)}{\zeta(2s)N^{2s-1}}(I_2(s;2) - I_0(s;2)) + \frac{s}{s+1}h_T\bigg(\frac{is}{2}\bigg)\varphi_{\infty\infty}\bigg(\frac{1+s}{2}\bigg) \\[2mm] 
    &+ \zeta(s)\bigg(\int_0^{\infty}\Psi_2^{\pm}(y)y^{s-1}\mathrm{d}y - \int_0^{\infty}\Psi_0^{\pm}(y)y^{s-1}\mathrm{d}y\bigg) + s\bigg(\frac{s-1}{2}\bigg)\int_0^{\infty}p_4^*(y)y^{s-2}\mathrm{d}y.
\end{align*}
Furthermore, the constant in the Laurent expansion at $s=1$ of the previous expression is 
\begin{align*}
    \mathscr{R}_{\infty}^{\mathrm{par}} &= \frac{12}{\pi N}\bigg(C_1(T) + A_2(T) \bigg(2\mathscr{C} + \gamma_{\!_{\mathrm{EM}}} - 2\log(N) \bigg)\bigg) + \frac{1-\log(4\pi)}{4\pi}  \\[2mm] 
    &+ \gamma_{\!_{\mathrm{EM}}}C_2(T) + C_3(T)+ \frac{\mathscr{C}_{\infty\infty}}{2} + \frac{v_{\Gamma}^{-1}}{2}(T+1) + C_4(T).
\end{align*}
\end{proposition}

\begin{proof}
By summing up the Laurent expansions for each $\mathcal{M}_j(s)$ at $s=1$ proved in the previous lemmas, 
%\begin{align*}
%    & \mathcal{M}_1(s) = \frac{12A_2(T)}{\pi N(s-1)} + \frac{12}{\pi N}\bigg(C_1(T) + A_2(T)\bigg(2\mathscr{C}+\gamma_{\!_{\mathrm{EM}}}-2\log(N)\bigg) \bigg) + O(s-1),\\[2mm]
%    & \mathcal{M}_2(s) = \frac{C_2(T) + (4\pi)^{-1}}{s-1} + \bigg(\frac{1- \log(4\pi)}{4\pi} + \gamma_{\textnormal{\tiny EM}}C_2(T) + C_3(T)\bigg) + O(s-1),\\[2mm]
%    & \mathcal{M}_3(s) = \frac{v_{\Gamma}^{-1}}{s-1} + \bigg( \frac{\mathscr{C}_{\infty\infty}}{2} + \frac{v_{\Gamma}^{-1}}{2}(T+1) \bigg) + O(s-1),\\[4mm]
%    &\mathcal{M}_4(s) = C_4(T) + O(s-1).
%\end{align*}
we can immediately deduce the constant term $\mathscr{R}_{\infty}^{\mathrm{par}}$, namely we have 
\begin{align*}
    \mathscr{R}_{\infty}^{\mathrm{par}} &= \frac{12}{\pi N}\bigg(C_1(T) + A_2(T)\bigg(2\mathscr{C}+\gamma_{\!_{\mathrm{EM}}}-2\log(N)\bigg) \bigg) + \frac{1- \log(4\pi)}{4\pi}\\ 
    & + \gamma_{\!_{\mathrm{EM}}}C_2(T) + C_3(T) + \frac{\mathscr{C}_{\infty\infty}}{2} + \frac{v_{\Gamma}^{-1}}{2}(T+1) + C_4(T).
\end{align*}
The result follows after noting that $$\frac{N\varphi(N)}{v_{\Gamma}}\prod_{p|N}\bigg(1+\frac{1}{p}\bigg) = \frac{6}{\pi N}$$ and reordering terms. This concludes the proof. 
\end{proof}

%=================%
%== SECTION SIX ==%
%=================%

\section{The self-intersection of the relative dualizing sheaf}
Let $\omega \coloneqq \omega_{\mathcal{X}/S}$ be the relative dualizing sheaf of $\mathcal{X}/S$, write $\overline{\omega}$ for the relative dualizing sheaf $\omega$, equipped with the Arakelov metric, and let $\overline{\omega}^2 \coloneqq \overline{\omega}_{\mathcal{X}/S}^2$ denote its self-intersection. 
In this chapter, we establish the main result of our article, namely an asymptotics for the  invariant 
\begin{align*}
e(\Gamma) \coloneqq \frac{1}{\varphi(N)}\overline{\omega}^2,
\end{align*}
as $N\to\infty$.
To prove this asymptotics, we proceed in two steps.  In section \ref{section6.1}, we first obtain an explicit formula for $\overline{\omega}^2$ in terms of an geometric contribution $\mathcal{G}(N)$ and an analytic contribution $\mathcal{A}(N)$ (see Proposition \ref{prop:self-intersection explicit formula}).  In section \ref{section6.2}, we then establish asymptotics for
 the geometric and analytic contribution (see Propositions \ref{asymptotics:geometric part}
and \ref{id:asymptotics of analytical contrib}), and we conclude with the main result 
in Theorem \ref{id:main result of the paper}.

\subsection{Explicit formula for the self-intersection}\label{section6.1}

Let $H_0$ resp.~$H_{\infty}$ be the horizontal divisor on $\mathcal{X}/S$ defined by the cusps 0 and $\infty$, respectively. For a prime ideal $\mathfrak{p}\in{}S$ such that $\mathfrak{p}|p$, where $p|N$ is a prime number, we set 
\begin{align*}
r_{\mathfrak{p}} &\coloneqq p+1, \\
s_{\mathfrak{p}} &\coloneqq \frac{p-1}{24}[\mathrm{PSL}_2(\mathbbm{Z}):\overline{\Gamma}(N/p)]. 
\end{align*}

Note that $s_{\mathfrak{p}}$ is the number of supersingular points on the fiber $\mathcal{X}_{\mathfrak{p}}$. Let us write $C_{1,\mathfrak{p}},\ldots,C_{r_{\mathfrak{p}},\mathfrak{p}}$ for the $r_{\mathfrak{p}}$ irreducible components of the fiber $\mathcal{X}_{\mathfrak{p}}$ stated in Proposition \ref{prop:Model description}, and $C_{0,\mathfrak{p}}$ resp. $C_{\infty,\mathfrak{p}}$ for the irreducible component of $\mathcal{X}_{\mathfrak{p}}$ intersected by $H_0$ and $H_{\infty}$, respectively.
\vskip 2mm

Consider the following divisors on $\mathcal{X}$ with rational coefficients 
\begin{align*}
V_0 \coloneqq -\sum_{\mathfrak{p}|N}\frac{2(g_{\Gamma}-1)}{r_{\mathfrak{p}}s_{\mathfrak{p}}}C_{0,\mathfrak{p}}, 
\qquad  
V_{\infty} \coloneqq -\sum_{\mathfrak{p}|N}\frac{2(g_{\Gamma}-1)}{r_{\mathfrak{p}}s_{\mathfrak{p}}}C_{\infty,\mathfrak{p}},
\end{align*}
where both sums run over all prime ideals $\mathfrak{p}\in S$ satisfying $\mathfrak{p}|p$ for some prime number $p|N$. Now, for $q\in\{0,\infty\}$, we define the admissible line bundle $\overline{\mathcal{L}}_q$ on $\mathcal{X}$ as follows
\begin{align*}
\overline{\mathcal{L}}_q \coloneqq \overline{\omega} \otimes \overline{\mathcal{O}(H_q)}^{\otimes-(2g_{\Gamma}-2)} \otimes \overline{\mathcal{O}(V_q)}.
\end{align*}
In addition, we define the divisor $\mathcal{M}$ on $\mathcal{X}$ by
\begin{align*}
\mathcal{M} \coloneqq H_{\infty} - H_{0} + \frac{1}{2g_{\Gamma}-2}(V_0-V_{\infty}).
\end{align*}

\begin{lemma}
\label{lem:self intersection vertical divisors}
Let $N\geq 3$ be a composite, odd, and square-free integer. Then, the identities
\begin{align}
\label{id:fundamental identities geometric part}
\begin{split}
(\overline{\mathcal{L}}_q, \overline{\mathcal{O}(V)})_{\mathrm{Ar}} &= 0, \\
(\mathcal{M},V)_{\mathrm{Ar}} &= 0,
\end{split}
\end{align}
hold for all vertical divisors $V$ of $\mathcal{X}$ and $q\in\{0,\infty\}$. Furthermore, we have
\begin{align*}
	(V_{\infty},V_{\infty})_{\mathrm{fin}} &= (V_{0},V_{0})_{\mathrm{fin}} = -4(g_{\Gamma}-1)\varphi(N)\bigg(1-\frac{6}{N}\bigg)\sum_{p|N}\frac{p^2 \log(p)}{p^2-1}, \\
	(V_{0},V_{\infty})_{\mathrm{fin}} &= 4(g_{\Gamma}-1)\varphi(N)\bigg(1-\frac{6}{N}\bigg)\sum_{p|N}\frac{p\log(p)}{p^2-1}.
\end{align*}
\end{lemma}

\begin{proof}
Throughout the proof, we let $\mathfrak{p},\tilde{\mathfrak{p}}\in{}S$ be prime ideals satisfying $\mathfrak{p}|N$ and $\tilde{\mathfrak{p}}\nmid N$, respectively.
\vskip 2mm

First of all, we claim that the identity $$(\overline{\mathcal{L}}_q, \overline{\mathcal{O}(V)})_{\mathrm{Ar}} = 0$$ holds for $V=\mathcal{X}_{\tilde{\mathfrak{p}}}$ and for $V=C_{q,\mathfrak{p}}$. Indeed, note that, from the definition of $\overline{\mathcal{L}}_q$, we have 
\begin{align*}
(\overline{\mathcal{L}}_q, \overline{\mathcal{O}(V)})_{\mathrm{Ar}} &= (\overline{\omega}, \overline{\mathcal{O}(V)})_{\mathrm{fin}} - (2g_{\Gamma}-2)(\overline{\mathcal{O}(H_q)},\overline{\mathcal{O}(V)})_{\mathrm{fin}} + (\overline{\mathcal{O}(V_q)},\overline{\mathcal{O}(V)})_{\mathrm{fin}}.
\end{align*}
Then, the claim follows by a straighforward computation using the following identities
\begin{align*}
(\overline{\omega},\overline{\mathcal{O}(V)})_{\mathrm{fin}} 
&= 
\begin{cases} 
\displaystyle (2g_{\Gamma}-2)\log(\#k(\tilde{\mathfrak{p}})), & V=\mathcal{X}_{\tilde{\mathfrak{p}}}; 
\\[2mm] 
\displaystyle \frac{2g_{\Gamma}-2}{r_{\mathfrak{p}}}\log(\#k(\mathfrak{p})), & V=C_{q,\mathfrak{p}};
\end{cases} 
\\[2mm]
(\overline{\mathcal{O}(H_q)},\overline{\mathcal{O}(V)})_{\mathrm{fin}} 
&= 
\begin{cases} 
\displaystyle \log(\#k(\tilde{\mathfrak{p}})), & V=\mathcal{X}_{\tilde{\mathfrak{p}}}; 
\\[2mm] \displaystyle \log(\#k(\mathfrak{p})), & V=C_{q,\mathfrak{p}};
%\\[2mm] \displaystyle 0, & V=C_{q',\mathfrak{p}};
\end{cases} 
\\[2mm]
(\overline{\mathcal{O}(V_q)},\overline{\mathcal{O}(V)})_{\mathrm{fin}} &= 
\begin{cases} \displaystyle 0, & V=\mathcal{X}_{\tilde{\mathfrak{p}}}; 
\\[2mm] 
\displaystyle \frac{2(g_{\Gamma}-1)(r_{\mathfrak{p}}-1)}{r_{\mathfrak{p}}}\log(\#k(\mathfrak{p})), & V=C_{q,\mathfrak{p}};
%\\[2mm] 
%\displaystyle -\frac{2(g_{\Gamma}-1)}{r_{\mathfrak{p}}}%\log(\#k(\mathfrak{p})), & V=C_{q',\mathfrak{p}}
\end{cases}
\end{align*}
(see, e.g., \cite[Chapter 9]{Liu}). Now, since every vertical divisor $V$ of $\mathcal{X}$ can be written as a linear combination of $\mathcal{X}_{\tilde{\mathfrak{p}}}$ and $C_{q,\mathfrak{p}}$, with coefficients in $\mathbbm{Q}$, the first identity of \eqref{id:fundamental identities geometric part} follows.  Similarly, note that $(\mathcal{M},V)_{\mathrm{Ar}}=0$ holds for $V = \mathcal{X}_{\tilde{\mathfrak{p}}}$ and $V=C_{q,\mathfrak{p}}$. Furthermore, if $V=C_{q',\mathfrak{p}}$ with $q'\in\{0,\infty\}$ and $q\neq{}q'$, then we have
\begin{align*}
(\overline{\mathcal{O}(H_q)},\overline{\mathcal{O}(V)})_{\mathrm{fin}} &= 0, \\[2mm]
(\overline{\mathcal{O}(V_q)},\overline{\mathcal{O}(V)})_{\mathrm{fin}} &= -\frac{2(g_{\Gamma}-1)}{r_{\mathfrak{p}}}\log(\#k(\mathfrak{p})).
\end{align*}
This yield the first assertion of the lemma.
\vskip 2mm

Next, to prove the second part of the lemma, observe that 
\begin{align*}
(V_{0},V_{\infty})_{\mathrm{fin}} 
&= \sum_{\mathfrak{p}|N} \sum_{\mathfrak{p}'|N}\frac{4(g_{\Gamma}-1)^2}{r_{\mathfrak{p}}r_{\mathfrak{p}'}s_{\mathfrak{p}}s_{\mathfrak{p}'}} (C_{0,\mathfrak{p}},C_{\infty,\mathfrak{p}'})_{\mathrm{fin}}.
\end{align*}

Since we have $(C_{0,\mathfrak{p}},C_{\infty,\mathfrak{p}'})_{\mathrm{fin}} = 0$ provided that $\mathfrak{p} \neq \mathfrak{p}'$, and
\begin{align*}
(C_{q,\mathfrak{p}},C_{q',\mathfrak{p}})_{\mathrm{fin}} = \begin{cases} s_{\mathfrak{p}}\log(\#k(\mathfrak{p})), & q\neq{}q'; \\[2mm] -(r_{\mathfrak{p}}-1)s_{\mathfrak{p}}\log(\#k(\mathfrak{p})), & q=q'; \end{cases}
\end{align*}
where $q'\in\{0,\infty\}$ with $q\neq{}q'$, we get
\begin{align*}
(V_{0},V_{\infty})_{\mathrm{fin}} 
&= \sum_{\mathfrak{p}|N}\frac{4(g_{\Gamma}-1)^2}{r_{\mathfrak{p}}^2s_{\mathfrak{p}}}\log(\#k(\mathfrak{p})).
\end{align*}
Further, if $\mathfrak{p}|p$ for some prime number $p|N$, then we have $$s_{\mathfrak{p}} = \frac{N}{24p(p+1)}\prod_{p'|N}((p')^2-1);$$ therefore, we obtain
\begin{align*}
(V_{0},V_{\infty})_{\mathrm{fin}} 
&= \frac{4(g_{\Gamma}-1)}{N}\frac{24(g_{\Gamma}-1)}{\prod_{p'|N}((p')^2-1)}\sum_{p|N}\bigg(\frac{p}{p+1}\sum_{\mathfrak{p}|p}\log(\#k(\mathfrak{p}))\bigg)\\
&= 4(g_{\Gamma}-1)\varphi(N)\bigg(1-\frac{6}{N}\bigg) \sum_{p|N}\frac{p\log(p)}{p^2-1}.
\end{align*}
The other identities can be proved in a similar way. This concludes the proof.
\end{proof}

\begin{proposition}
\label{prop:self-intersection explicit formula}
Let $N\geq 3$ be a composite, odd, and square-free integer. Then, the following identity holds
\begin{align*}
	\overline{\omega}^2 = \mathcal{G}(N) + \mathcal{A}(N), 
\end{align*}
where 
\begin{align*}
	\mathcal{G}(N) &= \frac{2g_{\Gamma}(V_0,V_{\infty})_{\mathrm{fin}} - (V_0,V_0)_{\mathrm{fin}} - (V_{\infty},V_{\infty})_{\mathrm{fin}}}{2(g_{\Gamma}-1)}, \\
	\mathcal{A}(N) &= 4g_{\Gamma}(g_{\Gamma}-1)\sum_{\sigma:\mathbbm{Q}(\zeta_N) \hookrightarrow \mathbbm{C}} g_{\mathrm{Ar}}(0^{\sigma},\infty^{\sigma}). 
\end{align*} 
Here, the sum runs over all embeddings of $\mathbbm{Q}(\zeta_N)$ into $\mathbbm{C}$, and $0^{\sigma}$ resp.~$\infty^{\sigma}$ denote the image of $0,\infty$ in $\mathcal{X}_{\eta}(\mathbbm{Q}(\zeta_N))_{\sigma}$ under the embedding  $\sigma$, respectively.
\end{proposition}

\begin{proof}
Throughout the proof, we will write $\mathcal{E}^2\coloneqq (\mathcal{E},\mathcal{E})_{\mathrm{Ar}}$. Let $\mathcal{L}_q \coloneqq \omega \otimes \mathcal{O}(H_q)^{\otimes-(2g_{\Gamma}-2)}\otimes \mathcal{O}(V_q)$.
\vskip 2mm

First of all, note that the pullback $L_q \coloneqq \mathcal{L}_q\otimes_{\mathbbm{Z}[\zeta_N]}\mathbbm{Q}(\zeta_N)$ of $\mathcal{L}_q$ to the generic fiber $\mathcal{X}_{\eta}$ defines a line bundle on $X(N)/\mathbbm{Q}(\zeta_N)$, that is supported on the cusps (see \cite[Lemme 4.1.1]{AbbesUllmo}). By the theorem of Manin--Drinfeld, $L_q$ defines a torsion point of $\mathrm{Jac}(X(N))(\mathbbm{Q}(\zeta_N))$. Furthermore, the theorem of Faltings--Hriljac implies 
\begin{align*}
\overline{\mathcal{L}}_q^2 
&= -2\varphi(N)h_{\mathrm{NT}}(L_q),
\end{align*}

where $h_{\mathrm{NT}}(\cdot)$ denotes the N\'e{}ron--Tate height. Since the latter vanishes on torsion points of the Jacobian, we obtain that 
\begin{align*}
\overline{\mathcal{L}}_q^2=0.
\end{align*}
Similarly, the restriction of the divisor $\mathcal{M}$ to the generic fiber $\mathcal{X}_{\eta}$ is supported on the cusps of $X(N)/\mathbbm{Q}(\zeta_N)$; therefore, by using the same arguments of the previous paragraph, we obtain $\mathcal{M}^2=0$.
\vskip 2mm

Secondly, if we expand the left hand side of $\overline{\mathcal{L}}_q^2=0$ using the definition of $\overline{\mathcal{L}}_q$, and apply the identities \eqref{id:fundamental identities geometric part}, then we obtain 
\begin{align*}
\overline{\omega}^2 = 2(2g_{\Gamma}-2)(\overline{\omega},\overline{\mathcal{O}(H_q)})_{\mathrm{Ar}} - (2g_{\Gamma}-2)^2\overline{\mathcal{O}(H_q)}^2 - (\overline{\omega}\otimes\overline{\mathcal{O}(H_q)}^{\otimes-(2g_{\Gamma}-2)},\overline{\mathcal{O}(V_q)})_{\mathrm{Ar}},
\end{align*}
where $q\in\{0,\infty\}$. By virtue of the equalities
\begin{align*}
(\overline{\omega} \otimes \overline{\mathcal{O}(H_q)}^{\otimes-(2g_{\Gamma}-2)},\overline{\mathcal{O}(V_q)})_{\mathrm{Ar}} &= -\overline{\mathcal{O}(V_q)}^2, \\[2mm]
(\overline{\omega}\otimes\overline{\mathcal{O}(H_q)},\overline{\mathcal{O}(H_q)})_{\mathrm{Ar}} &= 0,
\end{align*}
we now have
\begin{align*}
\overline{\omega}^2 = -4g_{\Gamma}(g_{\Gamma}-1)\overline{\mathcal{O}(H_q)}^2 + \overline{\mathcal{O}(V_q)}^2,
\end{align*}
for each $q\in\{0,\infty\}$. Thus, by adding the resulting identities for $q=0$ and $q=\infty$, we get
\begin{align*}
\overline{\omega}^2 = -2g_{\Gamma}(g_{\Gamma}-1)(\overline{\mathcal{O}(H_0)}^2 + \overline{\mathcal{O}(H_{\infty})}^2) + \frac{1}{2}(\overline{\mathcal{O}(V_0)}^2 + \overline{\mathcal{O}(V_{\infty})}^2).
\end{align*}
\vskip 2mm

Thirdly, putting $D \coloneqq H_{\infty} - H_{0}$ and $E \coloneqq (1/(2g_{\Gamma}-2))(V_0-V_{\infty})$, we have $$D^2 = \mathcal{M}^2 - 2(\mathcal{M},E)_{\mathrm{Ar}} + E^2 = E^2,$$ which in turn implies 
\begin{align*}
\overline{\mathcal{O}(H_{\infty})}^2 + \overline{\mathcal{O}(H_{0})}^2 =& \frac{1}{(2g_{\Gamma}-2)^2}\bigg(\overline{\mathcal{O}(V_{0})}^2 -2(\overline{\mathcal{O}(V_0)},\overline{\mathcal{O}(V_{\infty})})_{\mathrm{Ar}} + \overline{\mathcal{O}(V_{\infty})}^2 \bigg) + 2(\overline{\mathcal{O}(H_{\infty})},\overline{\mathcal{O}(H_{0})})_{\mathrm{Ar}}.
\end{align*}
Therefore, we have
\begin{align*}
\overline{\omega}^2 = -4g_{\Gamma}(g_{\Gamma}-1)(\overline{\mathcal{O}(H_{\infty})},\overline{\mathcal{O}(H_{0})})_{\mathrm{Ar}} + \frac{2g_{\Gamma}(V_0,V_{\infty})_{\mathrm{fin}} - (V_0,V_0)_{\mathrm{fin}} - (V_{\infty},V_{\infty})_{\mathrm{fin}}}{2(g_{\Gamma}-1)}.
\end{align*}

Finally, since 
\begin{align*}
(\overline{\mathcal{O}(H_{\infty})},\overline{\mathcal{O}(H_{0})})_{\mathrm{Ar}} = - \sum_{\sigma:\mathbbm{Q}(\zeta_N)\hookrightarrow\mathbbm{C}} g_{\mathrm{Ar}}(0^{\sigma},\infty^{\sigma}),
\end{align*}
the result follows. This concludes the proof of the proposition.
\end{proof}

\subsection{Asymptotics for the self-intersection}\label{section6.2}
%\begin{proposition}
%Let $N\geq 3$ be a composite, odd and square-free integer such that $g_{\Gamma} \geq 2$. Then, the following identity holds
%\begin{align*}
%	g_{\mathrm{Ar}}(0_{\xi},\infty) = -2\pi\mathscr{C}_{0_{\xi}\infty} - \frac{2\pi}{v_{\Gamma}} + 4\pi\mathscr{R}_{\infty} + 2\pi\mathscr{G}
%\end{align*}
%with 
%\begin{align*}
%	\mathscr{C}_{0_{\xi}\infty} &= 2v_{\Gamma}^{-1}\bigg( \mathscr{C} - \sum_{p|N}\frac{p^2-p-1}{p^2-1}\log(p)\bigg) + \kappa(\tilde{\xi};1,N) \\
%	\mathscr{R}_{\infty} &= -\frac{v_{\Gamma}^{-1}}{2g_{\Gamma}}\lim_{s\rightarrow{}1}\bigg(\frac{Z'}{Z}(s) - \frac{1}{s-1}\bigg) + \frac{1-\log(4\pi)}{4\pi g_{\Gamma}} + \frac{\mathscr{C}_{\infty\infty}}{g_{\Gamma}} \\ 
%	& \hspace{1em} + \frac{12}{N\pi{}g_{\Gamma}}\bigg( C_1 - \mathscr{C} - \frac{\gamma_{\mathrm{EM}}}{2} + \log(N) \bigg)\\
%	\mathscr{C}_{\infty\infty} &= 2v_{\Gamma}^{-1}\bigg( \mathscr{C} - \log(N) -\sum_{p|N}\frac{p^2}{p^2-1}\log(p)\bigg)
%\end{align*}
%and $\mathscr{G}$ is a constant such that $\mathscr{G} = O(1/g_{\Gamma})$ as $N\rightarrow\infty$.
%\end{proposition}

\begin{proposition}
\label{asymptotics:geometric part}
Let $N\geq 3$ be a composite, odd, and square-free integer. Then, the following asymptotics holds
\begin{align*}
	\frac{1}{\varphi(N)}\mathcal{G}(N) = o(g_{\Gamma}\log(N)),
\end{align*}
as $N\to\infty$.
\end{proposition}

\begin{proof}
Indeed, using Lemma \ref{lem:self intersection vertical divisors} and Proposition \ref{prop:self-intersection explicit formula}, we have 
\begin{align*}
\frac{1}{\varphi(N)}\mathcal{G}(N) = 4\bigg(1-\frac{6}{N}\bigg)\bigg(\log(N) + g_{\Gamma}\sum_{p|N}\frac{p\log(p)}{p^2-1} + \sum_{p|N}\frac{\log(p)}{p^2-1}\bigg).
\end{align*}

Using the fact that $$\sum_{p|N}\frac{\log(p)}{p} = O(\log\log(N)),$$ as $N\to\infty$ (see, e.g., \cite{Bruijn-vanLint}), we can deduce the following asymptotics
\begin{align*}
\sum_{p|N}\frac{p\log(p)}{p^2-1} = O(\log\log(N)), \quad \sum_{p|N}\frac{\log(p)}{p^2-1} = O(\log\log(N)),
\end{align*}
as $N\rightarrow\infty$. This yields
\begin{align*}
\frac{1}{\varphi(N)}\frac{\mathcal{G}(N)}{g_{\Gamma}\log(N)} = \frac{4}{g_{\Gamma}\log(N)}\bigg(1-\frac{6}{N}\bigg)\bigg(\log(N)+O(g_{\Gamma}\log\log(N))\bigg),
\end{align*}
and the result follows, since the right hand side of the previous identity converges to zero as $N\to\infty$. This concludes the proof.
\end{proof}

\begin{lemma}
\label{lem:reduction step}
Let $N\geq 3$ be a composite, odd, and square-free integer. Then, the following identity holds $$\mathscr{R}_{0_{\xi}} = \mathscr{R}_{\infty},$$ for all $\xi \in U$.
\end{lemma}

\begin{proof}
By taking $\alpha = \begin{psmallmatrix} -mN & \xi \\ -\tilde{\xi} & 1 \end{psmallmatrix}$ with $\mathrm{det}(\alpha)=1$, one can verify that $0_{\xi} = \alpha^{-1}\infty$ and $\alpha^{-1}\Gamma\alpha = \Gamma$; therefore, from Section \ref{subsect:Arakelov metric}, we have that $F(\alpha z) = F(z)$ with $z\in\mathbbm{H}$. Thus, since the identity $E_{\infty}(\alpha z,s) = E_{0_{\xi}}(z,s)$ holds, we obtain $\mathcal{R}_{0_{\xi}}[F](s) = \mathcal{R}_{\infty}[F](s)$. Consequently, by comparing coefficients in the Laurent expansions of $\mathcal{R}_{0_{\xi}}[F](s)$ and $\mathcal{R}_{\infty}[F](s)$ at $s=1$, we get $\mathscr{R}_{0_{\xi}} = \mathscr{R}_{\infty}$. This concludes the proof.
%
%Taking $\alpha=\begin{psmallmatrix} 0&1\\-1&0 \end{psmallmatrix}$, we have $0=\alpha^{-1}\infty$ and $\alpha^{-1}\Gamma\alpha = \Gamma$. Further, since $E_{\infty}(\alpha z,s) = E_0(z,s)$ holds, we obtain $\mathcal{R}_{0}[F](s) = \mathcal{R}_{\infty}[F](s)$; in particular, $\mathscr{R}_{0} = \mathscr{R}_{\infty}$.
%
%\begin{enumerate}[(i)]
	%\item Step 1: Note that $F(\alpha z) = F(z)$ for $z\in\mathbbm{H}$ and $\alpha\in\mathrm{GL}_2^+(\mathbbm{Q})$ such that $\alpha^{-1}\overline{\Gamma}(N)\alpha = \overline{\Gamma}(N)$.
%	\item Step 2: Taking $\alpha=\begin{psmallmatrix} 0&1\\-1&0 \end{psmallmatrix}$, we have $0=\alpha^{-1}\infty$ and $\alpha^{-1}\overline{\Gamma}(N)\alpha = \overline{\Gamma}(N)$. Since $E_{\infty}(\alpha z,s) = E_0(z,s)$, we obtain $\mathcal{R}_{0}[F](s) = \mathcal{R}_{\infty}[F](s)$; therefore, $\mathscr{R}_{0} = \mathscr{R}_{\infty}$.
%	\item Step 3: Similarly, taking $\alpha = \begin{psmallmatrix} -mN & \xi \\ -\tilde{\xi} & 1 \end{psmallmatrix}$ with $\mathrm{det}(\alpha)=1$, then $0_{\xi} = \alpha^{-1}\infty$ and $\alpha^{-1}\overline{\Gamma}(N)\alpha = \overline{\Gamma}(N)$. Consequently, $\mathscr{R}_{0_{\xi}} = \mathscr{R}_{\infty}$.
%\end{enumerate}
\end{proof}

\begin{proposition}
\label{lem:R_infty}
Let $N\geq 3$ be a composite, odd, and square-free integer. Then, the following identity holds
\begin{align*}
\mathscr{R}_{\infty} &= \frac{1}{g_{\Gamma}} \bigg( \frac{v_{\Gamma}^{-1}}{2}\lim_{s\rightarrow 1}\bigg(\frac{Z_{\Gamma}'}{Z_{\Gamma}}(s) - \frac{1}{s-1}\bigg) + \frac{12}{\pi N}\bigg(C_1 - \mathscr{C} - \frac{\gamma_{\!_{\mathrm{EM}}}}{2} + \log(N) \bigg) \\[2mm]
&+ v_{\Gamma}^{-1}\bigg( \mathscr{C}+1-2\log(N) - \sum_{p|N}\frac{\log(p)}{p^2-1}\bigg) + \frac{1-\log(4\pi)}{4\pi} \bigg),
\end{align*}
where $ C_1 = \lim_{T\rightarrow \infty}C_1(T)$. 
\end{proposition}

\begin{proof}
From identity \eqref{eqn:automorphic interpretation}, we know that
\begin{align*}%\label{eqn:automorphic interpretation}
	\mathscr{R}_{\infty} = \frac{1}{g_{\Gamma}}\bigg( -\frac{v_{\Gamma}^{-1}}{2}\sum_{j=1}^{\infty}\frac{h_T(r_j)}{\lambda_j} + \mathscr{R}_{\infty}^{\mathrm{hyp}} + \mathscr{R}_{\infty}^{\mathrm{par}}\bigg),
\end{align*} 
where now, by virtue of Proposition \ref{prop:hyperbolic contribution} and Proposition \ref{prop:parabolic contribution}, we have
\begin{align*}
\mathscr{R}_{\infty}^{\mathrm{hyp}} = \frac{v_{\Gamma}^{-1}}{2}\bigg( \lim_{s\rightarrow 1}\bigg(\frac{Z_{\Gamma}'}{Z_{\Gamma}}(s) - \frac{1}{s-1}\bigg) - T + 1 + o(1)\bigg),
\end{align*}
as $T\rightarrow\infty$, and 
\begin{align*}
    \mathscr{R}_{\infty}^{\mathrm{par}} &= \frac{12}{\pi N}\bigg(C_1(T) + A_2(T) \bigg(2\mathscr{C} + \gamma_{\!_{\mathrm{EM}}} - 2\log(N) \bigg)\bigg) + \frac{1-\log(4\pi)}{4\pi}  \\[2mm] 
    & + \frac{v_{\Gamma}^{-1}}{2}(T+1) + \frac{\mathscr{C}_{\infty\infty}}{2} + \gamma_{\!_{\mathrm{EM}}}C_2(T) + C_3(T) + C_4(T).
\end{align*}
Since $$\mathscr{C}_{\infty\infty} = 2v_{\Gamma}^{-1}\bigg(\mathscr{C} - 2\log(N) -\sum_{p|N}\frac{\log(p)}{p^2-1}\bigg)$$ (see, e.g., \cite{GradosPippich}), one gets
\begin{align*}
\mathscr{R}_{\infty} &= \frac{1}{g_{\Gamma}}\bigg( \frac{v_{\Gamma}^{-1}}{2}\lim_{s\rightarrow 1}\bigg(\frac{Z_{\Gamma}'}{Z_{\Gamma}}(s) - \frac{1}{s-1} \bigg) + \frac{12}{\pi N}\bigg( C_1(T) + A_2(T)(2\mathscr{C} + \gamma_{\!_{\mathrm{EM}}} - 2\log(N))\bigg) \\[2mm]
&+ v_{\Gamma}^{-1}\bigg(\mathscr{C} + 1 - 2\log(N) - \sum_{p|N}\frac{\log(p)}{p^2-1}\bigg) + \frac{1-\log(4\pi)}{4\pi} + \vartheta(T)\bigg),
\end{align*}
where $$\vartheta(T) \coloneqq -\frac{v_{\Gamma}^{-1}}{2}\sum_{j=1}^{\infty}\frac{h_T(r_j)}{\lambda_j} + \gamma_{\!_{\mathrm{EM}}}C_2(T) + C_3(T) + C_4(T) + \frac{v_{\Gamma}^{-1}}{2}o_T(1),$$ as $T\rightarrow\infty$. 
\vskip 2mm

Now, since $h_T(r_j),C_2(T),C_3(T)$, and $C_4(T)$ tend to zero as $T\rightarrow\infty$ (see \cite{AbbesUllmo}), we have that  
$\vartheta(T)\rightarrow 0$. Consequently, we get 
\begin{align*}
\mathscr{R}_{\infty} &= \frac{1}{g_{\Gamma}} \bigg( \frac{v_{\Gamma}^{-1}}{2}\lim_{s\rightarrow 1}\bigg(\frac{Z_{\Gamma}'}{Z_{\Gamma}}(s) - \frac{1}{s-1}\bigg) + \frac{12}{\pi N}\bigg(C_1 - \mathscr{C} - \frac{\gamma_{\!_{\mathrm{EM}}}}{2} + \log(N) \bigg) \\[2mm]
&+ v_{\Gamma}^{-1}\bigg( \mathscr{C}+1-2\log(N) - \sum_{p|N}\frac{\log(p)}{p^2-1}\bigg) + \frac{1-\log(4\pi)}{4\pi} \bigg).
\end{align*}
This concludes the proof.
\end{proof}

\begin{proposition}
\label{id:asymptotics of analytical contrib}
Let $N\geq 3$ be a composite, odd, and square-free integer. Then, the following asymptotics holds
\begin{align*}
	\frac{1}{\varphi(N)}\mathcal{A}(N) = 2g_{\Gamma}\log(N) + o(g_{\Gamma}\log(N)),
\end{align*}
as $N\to\infty$.
\end{proposition}

\begin{proof}
By Proposition \ref{prop:self-intersection explicit formula}, we have 
\begin{align*}
\frac{1}{\varphi(N)}\mathcal{A}(N) &= \frac{4g_{\Gamma}(g_{\Gamma}-1)}{\varphi(N)}\sum_{\sigma:\mathbbm{Q}(\zeta_N)\hookrightarrow\mathbbm{C}}g_{\mathrm{Ar}}(0^{\sigma},\infty^{\sigma}) \\[2mm]
&= \frac{4g_{\Gamma}(g_{\Gamma}-1)}{\varphi(N)}\sum_{\xi\in{}U}g_{\mathrm{Ar}}(0_{\xi},\infty),
\end{align*}
where in the second identity we identified each embedding $\sigma:\mathbbm{Q}(\zeta_N)\hookrightarrow\mathbbm{C}$ with a unique element $\xi\in{}U$ via the assignment $\sigma(\zeta_N) = e^{2\pi{}i\xi/N}$, and used the isomorphism $\iota_{\sigma}$ from Proposition \ref{prop:iota iso}.
\vskip 2mm

Next, from the identity \eqref{id:Abbes--Ullmo formula} and Lemma \ref{lem:reduction step}, we write $g_{\mathrm{Ar}}(0_{\xi},\infty)$ as follows
\begin{align*}
g_{\mathrm{Ar}}(0_{
\xi},\infty) &= -2\pi\mathscr{C}_{0_{\xi}\infty} - \frac{2\pi}{v_{\Gamma}}+4\pi\mathscr{R}_{\infty} + \mathscr{G},
\end{align*}
where $\mathscr{G}$ is given by \eqref{def_mathscrG}, 
and $\mathscr{C}_{0_{\xi}\infty}$ is given by
\begin{align*}
\mathscr{C}_{0_{\xi}\infty} &= 2v_{\Gamma}^{-1}\bigg( \mathscr{C} - \log(N) + \sum_{p|N}\frac{p\log(p)}{p^2-1} \bigg) + \kappa(\tilde{\xi};1,N);
\end{align*}
we refer to \cite{GradosPippich} for a proof of the latter identity and for the definition of $\kappa(\tilde{\xi};1,N)$. Letting $C:= 4\pi(g_{\Gamma}-1)v_{\Gamma}^{-1}$, we obtain
\begin{align*}
\frac{1}{\varphi(N)}\mathcal{A}(N) &= 2Cg_{\Gamma}\log(N) - 2 g_{\Gamma}C\bigg( \mathscr{C} + \frac{1}{2} + \sum_{p|N}\frac{p\log(p)}{p^2-1}\bigg) - \frac{8\pi g_{\Gamma}(g_{\Gamma}-1)}{\varphi(N)}\sum_{\xi\in U}\kappa(\overline{\xi};1,N) \\[2mm] 
&+ 8\pi g_{\Gamma}(g_{\Gamma}-1)\mathscr{R}_{\infty} + 2\mathscr{G}g_{\Gamma}(g_{\Gamma}-1).
\end{align*}
Observe that, from the definition of $\kappa(\overline{\xi};1,N)$ and the orthogonality relations of Dirichlet characters, we obtain $$\sum_{\xi\in U}\kappa(\overline{\xi};1,N) = 0.$$ Therefore, since $$\sum_{p|N}\frac{p\log(p)}{p^2-1} = O(\log\log(N)),$$ as $N\rightarrow\infty$ (see the proof of Proposition \ref{asymptotics:geometric part}), we get the following asymptotics
\begin{align*}
\frac{1}{\varphi(N)}\mathcal{A}(N) &= 2Cg_{\Gamma}\log(N) + 8\pi g_{\Gamma}(g_{\Gamma}-1)\mathscr{R}_{\infty} + 2\mathscr{G}g_{\Gamma}(g_{\Gamma}-1) + O(g_{\Gamma}\log\log(N)), 
\end{align*}
as $N\rightarrow\infty$. 
\vskip 2mm

Now, we claim that $$8\pi g_{\Gamma}(g_{\Gamma}-1)\mathscr{R}_{\infty} + 2\mathscr{G}g_{\Gamma}(g_{\Gamma}-1) = o(g_{\Gamma}\log(N)),$$ as $N\to\infty$. Indeed, using Proposition \ref{lem:R_infty}, we have
\begin{align*}
8\pi g_{\Gamma}(g_{\Gamma}-1)\mathscr{R}_{\infty} &= C\lim_{s\rightarrow 1}\bigg( \frac{Z_{\Gamma}'}{Z_{\Gamma}}(s) - \frac{1}{s-1}\bigg) + 2C\bigg( \mathscr{C} + 1 -\log(N)-\sum_{p|N}\frac{\log(p)}{p^2-1}\bigg) \\[2mm]
&+ \frac{96(g_{\Gamma}-1)}{N}\bigg(C_1 - \bigg( \mathscr{C} + \frac{\gamma_{\!_{\mathrm{EM}}}}{2}-\log(N)\bigg)\bigg) + 2(g_{\Gamma}-1)(1-\log(4\pi)),
\end{align*}
and by virtue of \cite[p.~27]{JorgensonKramer}, namely $$\lim_{s\rightarrow 1}\bigg( \frac{Z_{\Gamma}'}{Z_{\Gamma}}(s) - \frac{1}{s-1}\bigg) = O(N^{\varepsilon}),$$ for $\varepsilon > 0$, as $N\rightarrow \infty$, the claim follows for a sufficiently small $\varepsilon$. 
\vskip 2mm

Finally, note that 
\begin{align*}
2Cg_{\Gamma}\log(N) = 2g_{\Gamma}\log(N) + o(g_{\Gamma}\log(N)).
\end{align*}
This concludes the proof of the proposition.
\end{proof}

Finally, we can prove the main result of this paper.

\begin{theorem}
\label{id:main result of the paper}
Let $N\geq 3$ be a composite, odd, and square-free integer. Then, the following asymptotics holds
\begin{align*}
e(\Gamma) = 2g_{\Gamma}\log(N) + o(g_{\Gamma}\log(N)),
\end{align*}
as $N\to\infty$.
\end{theorem}

\begin{proof}
Recalling that $e(\Gamma)=\overline{\omega}^2/\varphi(N)$ and using Proposition \ref{prop:self-intersection explicit formula}, we get
\begin{align*}
	e(\Gamma) = \frac{1}{\varphi(N)}\left(\mathcal{G}(N) + \mathcal{A}(N)\right).
\end{align*}
Substituting therein the asymptotics proven in Propositions \ref{asymptotics:geometric part}
and \ref{id:asymptotics of analytical contrib}, and adding up, yields the asserted asymptotics.
\end{proof}

\begin{comment}
%================%
%== REFERENCES ==%
%================%
\section{TODO}
\begin{itemize}
\item Add the acknowledgement. Mention also Kramer. Mention the GRK1800.
\item In the introduction: the literature has to be updated: add the results by the Indian mathematicians
\item In the introduction: we should mention our article for the scattering constants in the introduction
\item In the introduction: add everywhere the precise restrictions $N$
\item In the introduction: add for every math symbol its meaning in case that it is missiung, e.g. for the roots of unity
\item In the introduction: we should better the structure of the whole proof
\item Why do we write $\Gamma$ and not $\Gamma(N)$ throughout this article?
\item Should we better write $g_{X(N)}$ instead of  $g_{\Gamma}$? In the latter 
I do not like that the N is not visible.
\item definition of definition of $\kappa(\tilde{\xi};1,N)$ is missing
\end{itemize}

Further, on page 12, the identity 
\begin{align*}
\frac{1}{\varphi(N)}\mathcal{A}(N) &= 2Cg_{\Gamma}\log(N) - 2 g_{\Gamma}C\bigg( \mathscr{C} + \frac{1}{2} + \sum_{p|N}\frac{p\log(p)}{p^2-1}\bigg) - \frac{8\pi g_{\Gamma}(g_{\Gamma}-1)}{\varphi(N)}\sum_{\xi\in U}\kappa(\overline{\xi};1,N) \\[2mm] 
&+ 8\pi g_{\Gamma}(g_{\Gamma}-1)\mathscr{R}_{\infty} + 2\mathscr{G}g_{\Gamma}(g_{\Gamma}-1)
\end{align*}
has to be explained a little bit more, e.g., by stating the cardinality of $U$. Also in this proof one should add a reference to the definition of $U$.
\end{comment}

%================%
%== REFERENCES ==%
%================%

\def\bibindent{.1mm}

\end{document}